\documentclass[reqno,tbtags,a4paper,12pt]{amsart}

\usepackage[in]{fullpage}
\usepackage[matrix,arrow,curve]{xy}
\usepackage{tikz,tikz-cd}
\usetikzlibrary{decorations.pathreplacing}
\usetikzlibrary{decorations.markings}

\usepackage{amsmath,amssymb,amsthm,amsfonts,amscd}
\usepackage{graphics,graphicx}

\usepackage{mathrsfs}
\renewcommand{\mathscr}{\mathcal}

\newcommand{\Ind}{\mathop{\mathrm{Ind}}\nolimits}

\newcommand{\Mod}{\mathop{\mathrm{Mod}}\nolimits}

\newcommand{\Sp}{\mathrm{Sp}}
\newcommand{\SO}{\mathrm{SO}}
\newcommand{\SL}{\mathrm{SL}}
\newcommand{\GL}{\mathrm{GL}}
\newcommand{\bSp}{\mathbf{Sp}}
\newcommand{\bSL}{\mathbf{SL}}
\newcommand{\bGL}{\mathbf{GL}}

\newcommand{\Aut}{\mathop{\mathrm{Aut}}\nolimits}
\newcommand{\Out}{\mathop{\mathrm{Out}}\nolimits}

\newcommand{\ab}{\mathrm{ab}}
\newcommand{\alg}{\mathrm{alg}}
\newcommand{\Hom}{\mathop{\mathrm{Hom}}\nolimits}
\newcommand{\Tor}{\mathop{\mathrm{Tor}}\nolimits}
\newcommand{\Diff}{\mathop{\mathrm{Diff}}\nolimits}
\newcommand{\Ann}{\mathop{\mathrm{Ann}}\nolimits}

\newcommand{\I}{\mathcal{I}}
\newcommand{\K}{\mathcal{K}}

\newcommand{\tors}{\mathop{\mathrm{tors}}\nolimits}
\newcommand{\Tors}{\mathop{\mathrm{tors}}\nolimits}

\newcommand{\Z}{\mathbb{Z}}

\newcommand{\Q}{\mathbb{Q}}

\newcommand{\bV}{\mathbf{V}}

\newcommand{\CC}{\mathcal{C}}

\newcommand{\CF}{\mathcal{F}}
\newcommand{\CG}{\mathcal{G}}

\newcommand{\CP}{\mathcal{P}}

\newcommand{\CT}{\mathcal{T}}
\newcommand{\CU}{\mathcal{U}}
\newcommand{\CV}{\mathcal{V}}

\newcommand{\bG}{\mathbf{G}}

\newcommand{\fX}{\mathfrak{X}}
\newcommand{\fS}{\mathfrak{S}}
\newcommand{\fg}{\mathfrak{g}}
\newcommand{\fsp}{\mathfrak{sp}}

\newtheorem{thm}{Theorem}

\newtheorem{theorem}{Theorem}[section]
\newtheorem{propos}[theorem] {Proposition}
\newtheorem{cor}[theorem] {Corollary}
\newtheorem{lem}[theorem]{Lemma}

\theoremstyle{definition}

\newtheorem{defin}[theorem]{Definition}
\newtheorem{remark}[theorem]{Remark}

\numberwithin{equation}{section}

\author{Alexander A. Gaifullin}

\thanks{This work was supported by the Russian Science Foundation under grant no.~26-11-00052,\\ \texttt{https:/\!/rscf.ru/en/project/26-11-00052/}.}

\title{Finite generation, algebraicity, and representation stability for  homology of Torelli groups}
\date{}

\address{Steklov Mathematical Institute of Russian Academy of Sciences, Moscow, Russia}
\address{Lomonosov Moscow State University, Russia}


\begin{document}

\begin{abstract}
 We solve a long-standing problem of whether the homology groups of the Torelli subgroups $\I_g\le\Mod_g$ are finitely generated in stable range. Namely, we prove that the group $H_k( \I_g;\Z)$ is finitely generated, provided that $k\le g-2$. Two main ingredients of our approach are as follows. First, we show that the action of any symplectic transvection~$t_x\in\Sp_{2g}(\Z)$ on the homology of~$\I_g$ satisfies the following unipotency condition: $(t_x-1)^{k+1}H_k( \I_g;\Z)=0$. The proof of this fact relies on the study of the spectral sequence for the action of~$\I_g$ on the complex of homologous curves on~$\Sigma_g$, which is known to be $(g-3)$-acyclic by a recent result of Minahan. The second key ingredient is Tavgen's theorem  asserting that the group $\Sp_{2g}(\Z)$ is boundedly elementarily generated. For homology with coefficients in~$\Q$, we further prove that $H_k(\I_g;\Q)$ is an algebraic $\Sp_{2g}(\Z)$-representation in the same stable range $k\le g-2$.
 Kupers and Randal-Williams have obtained a conditional result: they computed the algebraic part of the rational cohomology of Torelli groups in stable range under the assumpition that the rational cohomology groups are finite-dimensional in this stable range. Our results turn this conditional computation into a precise theorem that describes the whole rational cohomology ring of Torelli groups in stable range. As further applications, we obtain the following two consequences. First, we prove Morita's conjecture asserting that the $\Sp_{2g}(\Z)$-invariant part of the rational cohomology of~$\I_g$ stabilizes to the polynomial ring $\Q[e_2,e_4,\ldots]$ in the even Miller--Morita--Mumford classes. Second, we prove the uniform representation stability for the series of groups~$\left\{ H_k\bigl(\I_g^1;\Q)\right\}_{g=1}^{\infty}$.
\end{abstract}

\maketitle

\section{Introduction}

\subsection{Finite generation of homology}
In the study of the mapping class group~$\Mod_g$ of the genus~$g$ oriented surface~$\Sigma_g$, a fundamental role is played by the \textit{Torelli subgroup} $\I_g\le\Mod_g$ consisting of those mapping classes that act trivially on~$H_1(\Sigma_g;\Z)$. It is an old well-known problem to find out which homology groups $H_k(\I_g;\Z)$ are finitely generated and which are not, see, for instance~\cite[Problem~2.9]{Kir97}, \cite[Problem~5.11]{Far06}. A particularly important  question is whether these groups are finitely generated in a range of degrees tending to infinity with~$g$. Hain~\cite[Sect.~4]{Hai06} mentioned a folk conjecture that $H_k(\I_g;\Z)$ is finitely generated when $k\le g-2$. Our first result is exactly the proof of this conjecture.

Along with the closed surface~$\Sigma_g$, we will consider a surface with one puncture~$\Sigma_{g,1}$ and a surface with one boundary component~$\Sigma_g^1$. To treat these three cases together, we write~$\Sigma_{g,p}^b$ for a genus~$g$ oriented surface with $p$ punctures and $b$ boundary components, and agree that we omit an index if it is zero. However, throughout the  paper we always impose the condition $b+p\le 1$. We denote by~$\Mod_{g,p}^b$ and~$\I_{g,p}^b$ the mapping class group and Torelli group for~$\Sigma_{g,p}^b$, respectively. (Note that $\Mod_g^1$ is the mapping class group for diffeomorphisms of~$\Sigma_g^1$ that fix every  point of the boundary.) For definitions and results concerning mapping class groups, see~\cite{FaMa12}.

\begin{thm}\label{thm_main}
Suppose that $k\le g-2$ and $b+p\le 1$. Then the abelian group~$H_k\bigl(\I_{g,p}^b;\Z\bigr)$ is finitely generated.
\end{thm}

Previously, the finite generation of homology was known only for $k=1$ and $g\ge 3$ (Johnson~\cite{Joh83}, \cite{Joh85b}) and for the second rational homology groups~$H_2(\I_{g,p}^b;\Q)$ when $g\ge 5$ (Minahan~\cite{Min23a} for $g\ge51$, and Minahan and Putman~\cite{MiPu25} for all $g\ge 5$). Note that even for $k=2$ our result is stronger than that of Minahan and Putman. Firstly, we prove finite generation of the second homology group starting from $g=4$ instead of $g=5$. Secondly, we do this with integer coefficients rather than rational ones.

Note that the homology groups of the Torelli groups are not finitely generated in general. The first infiniteness result for Torelli groups was obtained by McCullough and Miller~\cite{MCM86} who showed that $\I_2$ is not finitely generated. Mess~\cite{Mes92} proved that $\I_2$ is an infinitely generated free group, and hence has not finitely generated~$H_1$. The first general result of non‑finite generation of higher homology groups was obtained by Akita~\cite{Aki01} who showed that $\dim H_*(\I_g;\Q)=\infty$, provided that $g\ge 7$. The groups $H_3(\I_3;\Z)$ and~$H_4(\I_3;\Z)$ were shown to be not finitely generated by Johnson and Millson (unpublished, see~\cite{Mes92}) and Hain~\cite{Hai02}, respectively. Later, Bestvina, Bux, and Margalit~\cite{BBM10} proved that the cohomological dimension of~$\I_g$ is equal to $3g-5$, and the top homology group~$H_{3g-5}(\I_g;\Z)$ is not finitely generated. Finally, the author~\cite{Gai24} extended the non‑finite generation statement for $H_k(\I_g;\Z)$ to all $k$ in range $2g-3\le k\le 3g-6$.

\begin{remark}
 Vladimirov~\cite{Vla25}, \cite{Vla26} studied the subgroup $H_k^{\mathrm{ab,sep}}(\I_g;\Z)\le H_k(\I_g;\Z)$ generated by special homology classes---namely, abelian cycles corresponding to $k$-tuples of pairwise disjoint separating simple closed curves on~$\Sigma_g$. He proved that this subgroup is a finite-dimensional vector space over~$\Z/2\Z$ in exactly the same range $k\le g-2$ as in Theorem~\ref{thm_main}. Although Vladimirov's result is not used in our proof of Theorem~\ref{thm_main}, it provided significant motivation for arriving at it.
\end{remark}

\subsection{Algebraicity}
We put $H=H_1(\Sigma_{g,p}^b;\Z)$ and choose an isomorphism $H\cong\Z^{2g}$ that takes the intersection form to the standard unimodular skew-symmetric form~$(\cdot,\cdot)$ on~$\Z^{2g}$. (Such an isomorphism exists, since $b+p\le 1$.) The action of the mapping class group~$\Mod_{g,p}^b$ on~$H$ preserves the intersection form and hence defines a short exact sequence
\begin{equation*}
 1\longrightarrow \I_{g,p}^b\longrightarrow\Mod_{g,p}^b\longrightarrow\Sp_{2g}(\Z)\longrightarrow 1.
\end{equation*}
Hence all homology groups $H_k(\I_{g,p}^b;\Z)$ are $\Sp_{2g}(\Z)$-modules.

We recall some facts on arithmetic groups that we will need, see~\cite{Ser79}, \cite{PlRa94}. We always consider the following setup:
\begin{enumerate}
 \item $\bG$ is a connected semisimple algebraic subgroup of~$\bGL_n$ defined over~$\Q$,
 \item $\Gamma$ is a Zariski dense arithmetic subgroup of its rational points~$\bG(\Q)$, where arithmeticity means that $\Gamma$ is commensurable with $\bG(\Z)=\bG(\Q)\cap\GL_n(\Z)$.
\end{enumerate}

Recall that a \textit{rational representation} of~$\bG$ on a finite-dimensional $\Q$-vector space~$V$ is a morphism of algebraic groups $\bG\to\bGL(V)$. A representation $\rho\colon\Gamma\to\GL(V)$ is called \textit{algebraic} if it can be obtained from a rational representation of~$\bG$ (defined over~$\Q$) by taking the rational points and then restricting to~$\Gamma$. (Note that an extension of~$\rho$ to a rational representation of~$\bG$ is unique whenever it exists, since $\Gamma$ is Zariski dense in~$\bG(\Q)$.) More generally, a representation~$\rho$ is called \textit{almost algebraic} if there exists a finite-index subgroup $\Lambda\le\Gamma$ such that the restriction of~$\rho$ to~$\Lambda$ is algebraic.

In this paper, we will only be interested in the case $\bG=\bSp_{2g}$ and $\Gamma=\Sp_{2g}(\Z)$ with $g\ge 2$. In this case, Bass, Milnor, and Serre~\cite[Theorem~16.2]{BMS67} proved that any representation of~$\Gamma$ on a finite-dimensional vector space over~$\Q$ is almost algebraic.

Recall that over a field of characteristic~$0$ all rational representations of~$\bSp_{2g}$ are semisimple and the irreducible rational representations of~$\bSp_{2g}$ are indexed by partitions $\lambda=[\lambda_1\ldots\lambda_m]$ of length~$m\le g$, where $\lambda_1\ge \cdots\ge\lambda_m\ge 1$; see~\cite{Hum75}, \cite{FuHa91}, and also a nice summary in~\cite[Sect.~7]{Put26}. We denote by $\bV_{\lambda}$ or $\bV_{\lambda_1,\ldots,\lambda_m}$ the irreducible representation that corresponds to a partition~$\lambda$. When we need to indicate the dependence of~$g$ in notation, we write~$\bV_{\lambda}(g)$ instead of~$\bV_{\lambda}$.

It is an important question whether the representations of~$\Sp_{2g}(\Z)$ on the rational homology groups~$H_k(\I_{g,p}^b;\Q)$ are algebraic without having to pass to a finite-index subgroup. The algebraicity was known for $H_1(\I_{g,p}^b;\Q)$ with $g\ge 3$ by a result of Johnson~\cite{Joh85b}. Minahan and Putman~\cite{MiPu25} have recently proved that $H_2(\I_{g,p}^b;\Q)$ is algebraic for $g\ge 6$. Our next result is as follows.

\begin{thm}\label{thm_alg_rep}
 Suppose that $k\le g-2$ and $b+p\le 1$. Then $H_k\bigl(\I_{g,p}^b;\Q\bigr)$, and hence $H^k\bigl(\I_{g,p}^b;\Q\bigr)$, is an algebraic representation of\/~$\Sp_{2g}(\Z)$. Moreover, $H_k\bigl(\I_{g,p}^b;\Q\bigr)$ splits into a direct sum of representations~$\bV_{\lambda_1,\ldots,\lambda_m}$ with $\lambda_1\le k$.
\end{thm}

\subsection{Explicit description of the cohomology ring}

Kupers and  Randal-Williams~\cite{KuRW20} obtained the following conditional result: they computed the algebraic part $H^*\bigl(\I_g^1;\Q\bigr)^{\alg}$ of the rational cohomology ring $H^*\bigl(\I_g^1;\Q\bigr)$ in stable range $*\ll g$ \textit{under the assumption} that the spaces $H^*\bigl(\I_g^1;\Q\bigr)$ are finite-dimensional in stable range.
Theorem~\ref{thm_main} ensures that this assumption holds, and Theorem~\ref{thm_alg_rep} shows that the result of Kupers and  Randal-Williams gives a description of the \textit{whole} rational cohomology ring $H^*\bigl(\I_g^1;\Q\bigr)$ in stable range $*\ll g$. The stable range in which the explicit description holds was incorrectly determined in~\cite{KuRW20}. A correction is given in~\cite{KuRW26}: it is shown there that the aforementioned description holds for $*\le \frac{g-4}4$.

To state the result thus obtained we need the construction of twisted Miller--Morita--Mumford classes. Let $\Diff(\Sigma_g,D)$ denote the topological group of diffeomorphisms of~$\Sigma_g$ that are equal to the identity near a specified two-dimensional disk $D\subseteq \Sigma_g$; we endow this group with the $C^{\infty}$-topology. Let $\Tor(\Sigma_g,D)\le \Diff(\Sigma_g,D)$ be the corresponding (topological) Torelli group consisting of all diffeomorphisms that act trivially on~$H_1(\Sigma_g;\Z)$. Let
$$
\Sigma_g\stackrel{f}{\longrightarrow} E\stackrel{\pi}{\longrightarrow} B\Tor(\Sigma_g,D)
$$
be the universal $\Sigma_g$-bundle over $B\Tor(\Sigma_g,D)$, let $T_{\pi}E\to E$ be its vertical tangent bundle, and let $s\colon B\Tor(\Sigma_g,D)\to E$ be the section of~$\pi$ determined by the centre of the disk~$D$. By studying the Serre spectral sequence for~$\pi$, one can prove (see~\cite[Sect.~1.2]{KuRW20} for details) that the map $f^*\colon H^1(E;\Q)\to H^1(\Sigma_g;\Q)$ is a surjection and has a unique right inverse $\iota\colon H^1(E;\Q)\to H^1(\Sigma_g;\Q)$ satisfying $s^*\circ \iota=0$. Then for $m\in\Z_{\ge0}$ and $v_1,\ldots,v_r\in H^1(\Sigma_g;\Q)$, the \textit{twisted Miller--Morita--Mumford classes} are defined by
\begin{equation}\label{eq_MMM}
\kappa_{e^m}(v_1\wedge \cdots\wedge v_r)=\pi_!\bigl(e^m(T_{\pi}E)\iota(v_1)\cdots\iota(v_r)\bigr)\in H^{2m+r-2}\bigl(B\Tor(\Sigma_g,D);\Q\bigr),
\end{equation}
where $e$ denotes the Euler class and $\pi_!$ is the Gysin homomorphism for~$\pi$. Finally, it is known (see~\cite{EaEe69}) that the connected component of the identity in the group~$\Diff(\Sigma_g,D)$ is contractible. Hence $B\Tor(\Sigma_g,D)\simeq B\I_g^1$ and we obtain the classes
$$
\kappa_{e^m}(v_1\wedge \cdots\wedge v_r)\in H^{2m+r-2}\bigl(\I_g^1;\Q\bigr).
$$

Combining Theorem~A from~\cite{KuRW20} with our Theorems~\ref{thm_main} and~\ref{thm_alg_rep} and taking into account the corrections made in~\cite{KuRW26}, we immediately obtain the following result.

\begin{thm}\label{thm_explicit}
 In the stable range of degrees $*\le \frac{g-4}4$, the graded-commutative ring $H^*\bigl(\I_g^1;\Q\bigr)$ is generated by the classes
 $$
 \kappa_{e^m}(v_1\wedge\cdots\wedge v_r)\qquad\text{with }m\ge0,\,r\ge0,\text{ and\/ }2m+r-2>0.
 $$
 A complete set of relations in this range is given by
 \begin{enumerate}
  \item linearity in each~$v_i$,
  \item skew-symmetry: $\kappa_{e^m}(v_{\sigma(1)}\wedge\cdots\wedge v_{\sigma(r)})=\mathrm{sign}(\sigma)\cdot \kappa_{e^m}(v_1\wedge\cdots\wedge v_r)$,
  \item $\sum_{i=1}^g\kappa_{e^m}(v_1\wedge\cdots\wedge v_k\wedge a_i)\,\kappa_{e^n}(a_i^{\#}\wedge v_{k+1}\wedge\cdots\wedge v_r)=\kappa_{e^{m+n}}(v_1\wedge\cdots\wedge v_r)$,
  \item $\sum_{i=1}^g\kappa_{e^m}(v_1\wedge\cdots\wedge v_r\wedge a_i\wedge a_i^{\#})=\kappa_{e^{m+1}}(v_1\wedge\cdots\wedge v_r)$,
  \item $\kappa_{e^{2n}}(1)=0$,
 \end{enumerate}
 where $a_1,\ldots,a_{2g}$ is a basis of $H^1(\Sigma_g;\Q)$ and $a_1^{\#},\ldots,a_{2g}^{\#}$ is the Poincar\'e dual basis characterized by $\bigl\langle a_i^{\#}a_j,[\Sigma_g]\bigr\rangle=\delta_{ij}$.
\end{thm}

\begin{remark}
 We write $\kappa_{e^m}(v_1\wedge\cdots\wedge v_r)$ instead of $\kappa_{e^m}(v_1\otimes\cdots\otimes v_r)$ as  in~\cite{KuRW20}, to emphasize the skew-symmetry.
\end{remark}

\begin{remark}\label{rem_MMM}
 For $r=0$, we will usually abbreviate the notation~$\kappa_{e^{m}}(1)$ simply to~$\kappa_{e^m}$. The classes~$\kappa_{e^m}$ are the usual (untwisted) Miller--Morita--Mumford classes; the standard notation for them is $e_m=\kappa_{e^{m+1}}$. These classes are defined by the same formula~\eqref{eq_MMM} in the cohomology of the whole mapping class group~$\Mod_g$. It is a standard fact (see~\cite[Sect.~2]{Mor87}, \cite[Sect.~3]{Mor99}) that the odd Miller--Morita--Mumford classes are pullbacks of certain cohomology classes of~$\Sp_{2g}(\Z)$ and hence vanish when restricted to the Torelli group. This is exactly relation~(5) from Theorem~\ref{thm_explicit}. In~\cite{KuRW20} this relation is written in the form~$\kappa_{\mathcal{L}_n}=0$, where $\mathcal{L}_n$ denotes the $n$th Hirzebruch $\mathcal{L}$-class of~$T_{\pi}E$. These are the same relations, since for a rank~$2$ oriented real vector bundle one has $\mathcal{L}_n=\frac{2^{2n}B_{2n}}{(2n)!}e^{2n}$, where $B_{2n}$ are the Bernoulli numbers.
\end{remark}

\begin{remark}
Similar presentations can also be written for the rings~$H^*(\I_g;\Z)$ and $H^*(\I_{g,1};\Z)$, using Theorem~4.1 from~\cite{RW23} instead of Theorem~A from~\cite{KuRW20}, together with the same correction from~\cite{KuRW26}. We do not include them here.
\end{remark}

As a consequence of Theorem~\ref{thm_explicit}, an explicit calculation of the $\Sp_{2g}(\Z)$-representation $H^3(\I_g^1;\Q)^{\alg}$ from~\cite[Theorem~8.1]{KuRW20} and~\cite[Section~2]{KuRW26} also becomes a result about the whole group~$H^3(\I_g^1;\Q)$. (A similar result for the second rational homology of~$\I_g^1$ was already known by the algebraicity result of Minahan and Putman in this case, see~\cite[Theorem~A]{MiPu25}.)

\begin{cor}
 For $g\ge16$, we have
 \begin{align*}
 H^3\bigl(\I_g^1;\Q\bigr)\cong \bV_1+\bV_{2,1}+3\bV_{1^3}+2\bV_{2^2,1}+3\bV_{2,1^3}+\bV_{3,2,1^2}\\
 {}+2\bV_{2^3,1}+\bV_{3,2^3}+4\bV_{1^5}+2\bV_{2^2,1^3}+\bV_{3^2,1^3}\\
 {}+2\bV_{2,1^5}+\bV_{2^3,1^3}+2\bV_{1^7}+\bV_{2^2,1^5}+\bV_{1^9}.
 \end{align*}
\end{cor}

After we prove finite-dimensionality and algebraicity of~$H^*(\I_g^1;\Q)$ in stable range, and thus establish the presentation from Theorem~\ref{thm_explicit}, we can obtain further consequences from the results of~\cite{KuRW20}, which were originally \textit{conditional} results about the maximal algebraic subrepresentation~$H^*(\I_g^1;\Q)^{\alg}$ (under the condition of finite dimensionality), but now become \textit{unconditional} results about the whole ring~$H^*(\I_g^1;\Q)$.
In the next two subsections we obtain two such consequences: we prove Morita's conjecture on the invariant part of the cohomology of Torelli groups and the uniform representation stability for $\left\{H_k(\I_g^1;\Q)\right\}_{g=1}^{\infty}$.

\subsection{Morita's conjecture}

A remarkable theorem by Madsen and Weiss~\cite{MaWe07} says that
$$
\lim_{g\to\infty}H^*\bigl(\Mod_g^1;\Q\bigr)=\Q[e_1,e_2,e_3,\ldots],
$$
where $e_m$'s are the Miller--Morita--Mumford classes. As we have already mentioned (see Remark~\ref{rem_MMM}), the odd Miller--Morita--Mumford classes~$e_{2n+1}$ vanish when restricted to the Torelli group. An old conjecture due to Morita, see~\cite[Conjecture~3.4]{Mor99} (cf.~\cite[Problem~3.1]{Mad06}) asserts that any class~$e_{2n}$ is nontrivial in the rational cohomology of~$\I_g$ (or~$\I_g^1$) for sufficiently large~$g$, and moreover, the $\Sp_{2g}(\Z)$-invariant part of the cohomology of~$\I_g^1$ stabilizes and we have an isomorphism
$$
\lim_{g\to\infty}H^*\bigl(\I_g^1;\Q\bigr)^{\Sp_{2g}(\Z)}=\Q[e_2,e_4,e_6,\ldots].
$$
We prove this conjecture.

\begin{thm}\label{thm_inv}
 Suppose that $b\in\{0,1\}$. Then in the stable range $*\le \frac{2g-2}3$, we have
 $$
 H^*\bigl(\I_g^b;\Q\bigr)^{\Sp_{2g}(\Z)}=\Q[e_2,e_4,e_6,\ldots].
 $$
 In particular, the even Miller--Morita--Mumford classes~$e_{2n}$ are nonzero, and all monomials in them are linearly independent in this stable range.
\end{thm}

\begin{remark}\label{rem_eps}
In the case of the group~$\I_{g,1}$ we have an additional $\Sp_{2g}(\Z)$-invariant class $\epsilon\in H^2(\I_{g,1};\Q)$. Namely, consider the fibre bundle
$$
\Sigma_g \longrightarrow E\stackrel{\pi}{\longrightarrow} B\Tor^+(\Sigma_g,*),
$$
where $\Tor^+(\Sigma_g,*)$ is the Torelli group for the group of orientation-preserving diffeomorphisms of~$\Sigma_g$ that fix the prescribed base point $*\in\Sigma_g$. Then the tangent spaces to fibres of~$\pi$ at the base point~$*$ provide a  rank~$2$ oriented real vector bundle~$\eta$ over $B\Tor^+(\Sigma_g,*)\simeq B\I_{g,1}$. Then $\epsilon$ is the Euler class of this bundle. The analog of Theorem~\ref{thm_inv} in this case is the equality
$$
H^*(\I_{g,1};\Q)^{\Sp_{2g}(\Z)}=\Q[\epsilon,e_2,e_4,e_6,\ldots]
$$
in the same stable range $*\le \frac{2g-2}3$.
\end{remark}

\subsection{Representation stability.}

In~\cite{ChFa13} Church and Farb introduced a concept of \textit{uniform representation stability}. In the case of $\Sp_{2g}(\Q)$-representations, the definition is as follows.

Let $\{V_g\}$ be a sequence of rational $\Sp_{2g}(\Q)$-representations equipped with linear maps $\varphi_g\colon V_g\to V_{g+1}$ that is \textit{compatible} in the sense that $\varphi_g(\gamma v)=f_g(\gamma)\varphi_g(v)$ for all $\gamma\in \Sp_{2g}(\Q)$ and $v\in V_g$, where $f_g$ is the standard inclusion $\Sp_{2g}(\Q)\hookrightarrow  \Sp_{2g+2}(\Q)$. The sequence $\{V_g\}$ is said to be \textit{uniformly representation stable} if there exists some~$N$ such that each of the following  conditions holds:

\begin{enumerate}
 \item \textbf{Injectivity:} The map $\varphi_g\colon V_g\to V_{g+1}$ is injective for all $g\ge N$,
 \item \textbf{Surjectivity:} The induced map $\Ind_{\Sp_{2g}(\Q)}^{\Sp_{2g+2}(\Q)}V_g\to V_{g+1}$ is surjective for all $g\ge N$,
 \item \textbf{Multiplicity stability:} Write
 $$
 V_g\cong\bigoplus_{\lambda}c_{\lambda,g}\bV_{\lambda}(g).
 $$
 Then the multiplicities~$c_{\lambda,g}$ are independent of~$g$ for all $g\ge N$. In particular, for any~$\lambda$ for which $\bV_{\lambda}(N)$ is not defined $c_{\lambda,g}=0$ for all $g\ge N$.
\end{enumerate}

The uniform representation stability for $\left\{ H_1(\I_g^1;\Q)\right\}_{g=1}^{\infty}$ follows immediately from a classical result of Johnson~\cite{Joh85b}. The uniform representation stability for $\left\{ H_2(\I_g^1;\Q)\right\}_{g=1}^{\infty}$ was proved by Minahan and Putman, see~\cite{MiPu25}, \cite{Put26}. We extend this result to all homology groups. By Theorem~\ref{thm_alg_rep} each space~$H_k(\I_g^1;\Q)$ with $k\le g-2$ becomes a rational $\Sp_{2g}(\Q)$-representation.

\begin{thm}\label{thm_rs}
 For each~$k$, the sequence $\left\{ H_k\bigl(\I_g^1;\Q\bigr)\right\}_{g=1}^{\infty}$ is uniformly representation stable starting from $g=6k+1$. Moreover, the injectivity and multiplicity stability properties hold starting from $g=4k+4$. Also, for $g\ge 4k+4$, the $\Sp_{2g}(\Q)$-representation $H_k\bigl(\I_g^1;\Q\bigr)$ decomposes into irreducible summands $\bV_{\lambda_1,\ldots,\lambda_m}(g)$, each of which satisfies $\lambda_1+\cdots+\lambda_m\le 3k$ and $\lambda_1\le k$.
\end{thm}

The sum $|\lambda|=\lambda_1+\cdots+\lambda_m$ is called the \textit{degree} of an irreducible representation~$\bV_{\lambda}(g)$; it is equal to the smallest number~$d$ such that $\bV_{\lambda}(g)$ appears in~$H(g)^{\otimes d}$, where $H(g)$ is the tautological representation of~$\Sp_{2g}(\Q)$.

The injectivity and multiplicity stability properties follow immediately from Theorems~\ref{thm_main} and~\ref{thm_alg_rep} and the results from~\cite{KuRW20} (with the corrections from~\cite{KuRW26}). Surjectivity, however, requires additional work, see Proposition~\ref{propos_theta}. Note that for $k=2$ the surjectivity property was already known from a result of Boldsen and Dollerup~\cite{BoDo12}, but for $k>2$ no such result is available. The author does not know whether it is possible to extract the surjectivity property starting from $g=4k+4$ from the results of~\cite{KuRW20} and~\cite{KuRW26}.

\subsection{Outline of the proofs of Theorems~\ref{thm_main} and~\ref{thm_alg_rep}}
For $x\in\Z^{2g}$, the transformation $t_x\in\Sp_{2g}(\Z)$ defined by
$$
t_xy=y+(x,y)x,
$$
as well as any of its powers~$t_x^{\xi}$, is called a \textit{symplectic transvection}.

Our approach is partially inspired by the author's proof (see~\cite{Gai25}) of the finite generation of the abelianization~$(\K_g^b)^{\ab}$ for $g\ge 3$ and $b\in\{0,1\}$, where $\K_g^b\le\I_g^b$ is the \textit{Johnson kernel}, i.\,e., the subgroup generated by all Dehn twists about separating simple closed curves.

A key role in our approach is played by the following property of $\Sp_{2g}(\Z)$-modules, which we shall call transvection unipotence.

\begin{defin}\label{defin_transv_unip}
 An $\Sp_{2g}(\Z)$-module~$M$ will be called \textit{transvection unipotent} if there exists a positive integer~$n$ such that $(t_x-1)^nM=0$ for all $x\in\Z^{2g}$. The smallest $n$ with this property will be called the \textit{transvection unipotency index} of~$M$.
\end{defin}

The main result on which Theorems~\ref{thm_main} and~\ref{thm_alg_rep} rely is as follows.

\begin{thm}\label{thm_annih}
 If $k\le g-2$ and $b+p\le 1$, then $H_k\bigl(\I_{g,p}^b;\Z\bigr)$ is a transvection unipotent
$\Sp_{2g}(\Z)$-module of  index $\le k+1$, that is,
\begin{equation}\label{eq_ucond}
(t_x-1)^{k+1}H_k\bigl(\I_{g,p}^b;\Z\bigr)=0
\end{equation}
for all $x\in H=\Z^{2g}$.
\end{thm}

The proof of this theorem is based on the study of a spectral sequence
\begin{equation}\label{eq_ss_intro}
 E^1_{p,q}=\bigoplus_{\sigma\in\fX_p}H_q(\I_{\sigma};\Z) \ \Longrightarrow \ H_{p+q}^{\I}\bigl(\CC_x(\Sigma_g);\Z\bigr)
\end{equation}
associated with the action of~$\I=\I_{g,p}^b$ on the \textit{complex of homologous curves}~$\CC_x(\Sigma_g)$ defined by Putman~\cite{Put08}. Here $\fX_p$ is a set of representatives for $\I$-orbits of $p$-cells of~$\CC_x(\Sigma_g)$. Minahan~\cite{Min23} proved that the complex~$\CC_x(\Sigma_g)$ is $(g-3)$-acyclic, which implies that the spectral sequence~\eqref{eq_ss_intro} converges to the homology of~$\I$ in dimensions $k<g-2$, and in dimension~$g-2$, the resulting limit group surjects onto $H_{g-2}(\I;\Z)$. We will prove that the transvection~$t_x$ acts trivially on the spectral sequence~\eqref{eq_ss_intro}. The required unipotence condition~\eqref{eq_ucond} will follow, since the $E^{\infty}$ page is the associated graded of some filtration in the homology of~$\I$.

It turns out that Theorem~\ref{thm_annih} is already sufficient to deduce Theorem~\ref{thm_alg_rep} from Theorem~\ref{thm_main}. Indeed, we will prove in Section~\ref{section_algebraic} that any finite-dimesional transvection unipotent representation of~$\Sp_{2g}(\Z)$ with $g\ge 2$ is algebraic.

The second ingredient in the proof of Theorem~\ref{thm_main} will be results on bounded elementary generation and finite width for the groups~$\Sp_{2g}(\Z)$, where $g\ge 2$.

\begin{defin}\label{defin_fw}
  A group $\Gamma$ is said to have \textit{finite width} with respect to a finite sequence of (not necessarily distinct) elements $\gamma_1,\ldots,\gamma_m$ if
  every element $\gamma\in\Gamma$ can be written in the form
 $$
 \gamma=\gamma_1^{k_1}\cdots\gamma_m^{k_m},\qquad k_i\in\Z,
 $$
 or equivalently, if $\Gamma$ equals the product of its cyclic subgroups:
$
\Gamma=\langle\gamma_1\rangle\cdots \langle\gamma_m\rangle.
$
\end{defin}

We need the following result that is due to Tavgen~\cite{Tav91}. (This result is stated in~\cite{Tav91} in a slightly different form; we will explain how to extract it from the results of that paper in Section~\ref{section_bounded}.)

\begin{propos}\label{propos_fw}
 If $g\ge 2$, then there exists a finite sequence of (not necessarily distinct) elements $x_1,\ldots,x_m\in\Z^{2g}$ such that the group~$\Sp_{2g}(\Z)$ has finite width with respect to the corresponding sequence of symplectic transvections: $\Sp_{2g}(\Z)=\langle t_{x_1}\rangle\cdots \langle t_{x_m}\rangle$.
\end{propos}

A key observation is that this proposition easily implies that any finitely generated transvection unipotent $\Sp_{2g}(\Z)$-module is finitely generated as an abelian group. Kassabov and Putman~\cite{KaPu20} proved that the groups $H_2(\I_{g,p}^b;\Z)$ are finitely generated as $\Sp_{2g}(\Z)$-modules. Combining this result with Theorem~\ref{thm_annih} and Proposition~\ref{propos_fw}, we obtain the assertion of Theorem~\ref{thm_main} for $k=2$.

However, in the case $k>2$ this argument does not work, because we initially have no result on the finite generation of the $\Sp_{2g}(\Z)$-module $H_k(\I_{g,p}^b;\Z)$. Therefore, we need to learn how to do without this result. We will do this as follows.

We proceed by induction on~$k$, so we assume that all homology groups~$H_i(\I_{g,p}^b;\Z)$ with $i<k$ are finitely generated. Raghunathan~\cite{Rag68} and Ivanov~\cite{Iva87} proved the finite generation of all homology groups (even with twisted coefficients) of the groups~$\Sp_{2g}(\Z)$ and~$\Mod_{g,p}^b$, respectively. Then studying the Lyndon--Hochschild--Serre spectral sequence for the short exact sequence
\begin{equation*}
 1\longrightarrow \I_{g,p}^b\longrightarrow\Mod_{g,p}^b\longrightarrow\Sp_{2g}(\Z)\longrightarrow 1,
\end{equation*}
we will prove that the $k$th homology group~$H_k(\I_{g,p}^b;\Z)$ satisfies the following condition:
\medskip

\textit{For any finite-index subgroup $\Lambda\le\Sp_{2g}(\Z)$ and any $\Lambda$-module~$A$ that is finitely generated as an abelian group, the coinvariant group $\left(H_k\bigl(\I_{g,p}^b;\Z\bigr)\otimes A\right)_{\Lambda}$ is finitely generated.}
\smallskip

In general, for an $\Sp_{2g}(\Z)$-module~$M$, the condition of having finitely generated coinvariants~$(M\otimes A)_{\Lambda}$ for all~$\Lambda$ and~$A$ as described above is much weaker than the condition of being finitely generated as an~$\Sp_{2g}(\Z)$-module. Nevertheless, we will prove that for transvection unipotent modules these two conditions are equivalent, so the finite generation of coinvariants suffices to show that $H_k\bigl(\I_{g,p}^b;\Z\bigr)$ is a finitely generated abelian group. This will complete the proof of Theorem~\ref{thm_main}.

\subsection{Structure of the paper}
Section~\ref{section_ss} gathers the necessary preliminaries on a spectral sequence associated with the action of a group on a CW complex. Section~\ref{section_annih} is the core of the paper: we study the spectral sequence for the action of the Torelli group on a complex of homologous curves and prove Theorem~\ref{thm_annih}. In Section~\ref{section_U} we construct universal $\CT$-unipotent modules and prove their finite generation in the case of groups of finite width. Section~\ref{section_bounded} contains the proof of Theorem~\ref{thm_main} for~$k=2$. The proof of Theorem~\ref{thm_main} for $k>2$ occupies Sections~\ref{section_coinv}--\ref{section_proof_main}. More precisely, in Section~\ref{section_coinv}   we prove the finite generation of the coinvariants of the homology of Torelli groups. Section~\ref{section_quasiproj} provides an auxiliary representation‑theoretic construction. In Section~\ref{section_arithm} we recall the necessary facts about arithmetic groups and their representations. Finally, in Section~\ref{section_proof_main} we obtain a general result of finite generation for $\CT$-unipotent $\Gamma$-modules with finitely generated coinvariants, where $\Gamma$ is an arithmetic group of finite width (Theorem~\ref{thm_alg}). As a result, we prove Theorem~\ref{thm_main} in the general case. Section~\ref{section_algebraic} contains the proof of Theorem~\ref{thm_alg_rep} on algebraicity of the rational homology of Torelli groups in the stable range. Finally, in Sections~\ref{section_rs} and~\ref{section_inv} we prove Theorems~\ref{thm_rs} and~\ref{thm_inv}, respectively.

\subsection{Conventions and notation} In this paper, we always work with left modules. Modules over the group ring~$\Z[G]$ will be called $G$-modules. The symbol~$\otimes$ will denote the tensor product over~$\Z$ (or, equivalently, over~$\Q$ for $\Q$-vector spaces) unless explicitly stated otherwise.

An element~$a$ of an abelian group~$A$ is called \textit{primitive} if it cannot be written in the form $nb$ where $b\in A$ and $n>1$.

Throughout the paper we consider only surfaces~$\Sigma_{g,p}^b$ with $b+p\le 1$.

\subsection{Acknowledgements}
The author is grateful to Denis Osipov, Vasilii Rozhdestvenskii, and Andrei Vladimirov for fruitful discussions.

\section{Spectral sequence}\label{section_ss}

All statements in this section hold for homology with coefficients in any commutative ring~$R$; therefore, we omit the coefficients from the notation.

Assume that a group~$G$ acts cellulary and without rotations on a regular CW complex~$X$. Then we have the following spectral sequence (see~\cite[(7.7)]{Bro82}):
\begin{equation}\label{eq_spec_seq}
 E^1_{p,q}=\bigoplus_{\sigma\in\fX_p}H_q(G_{\sigma}) \ \Longrightarrow \ H_{p+q}^G(X).
\end{equation}
Here $\fX_p$ is a set of representatives for $G$-orbits of $p$-cells of~$X$, $G_{\sigma}$ denotes the stabilizer of~$\sigma$, and $H^G_*(X)$ is the equivariant homology of~$X$. `Without rotations' means that if $g\in G$ fixes a cell~$\sigma$ setwise, then $g$ fixes every point of~$\sigma$.  The convergency of~\eqref{eq_spec_seq} means that, for each~$k$, there is a filtration
\begin{equation}\label{eq_filtration}
 0=\CF_{-1,k+1}\subseteq\CF_{0,k}\subseteq\CF_{1,k-1}\subseteq\cdots\subseteq\CF_{k,0}=H_k^G(X)
\end{equation}
such that $E^{\infty}_{p,q}=\CF_{p,q}/\CF_{p-1,q+1}$.

The spectral sequence~\eqref{eq_spec_seq} is functorial in the following sense (see~\cite[Ch.~VII]{Bro82} and also~\cite[Sect.~3.2]{Gai21}):

\begin{enumerate}
 \item The inclusion $\iota_{\sigma}\colon H_q(G_{\sigma})\hookrightarrow E^1_{p,q}$ is independent of the choice of representatives for all other $G$-orbits of cells of~$X$. Moreover, if we replace~$\sigma$ with another representative~$g\sigma$ of the same $G$-orbit, then we get a commutative diagram
 \begin{equation}\label{eq_cd_iota}
 \begin{tikzcd}
  H_q(G_{\sigma}) \arrow[rd,hookrightarrow,"{\iota_{\sigma}}"] \arrow[dd,"{g_*}","\cong"'] & \\
  & E^1_{p,q} \\
  H_q(G_{g\sigma}) \arrow[ru,hookrightarrow,"{\iota_{g\sigma}}"']
 \end{tikzcd}
 \end{equation}
 where the isomorphism~$g_*$ is induced by the conjugation by~$g$. In particular, the summands in the direct sum decompositions of~$E^1_{p,q}$ are independent of the choice of a set~$\fX_p$.

 \item Suppose that $G$ acts cellulary and without rotations on~$X$, $H$ acts cellulary and without rotations on~$Y$, $\varphi\colon G\to H$ is a homomorphism, and $f\colon X\to Y$ is a $\varphi$-equivariant map of CW complexes that maps every cell of~$X$ homeomorphically onto a cell of~$Y$. Then the pair $\theta=(\varphi,f)$ induces a morphism
 $$
 E^r_{p,q}(\theta)\colon E^r_{p,q}(G,X)\to E^r_{p,q}(H,Y)
 $$
 of the corresponding spectral sequences~\eqref{eq_spec_seq} such that
 \begin{itemize}
  \item $E^1_{p,q}(\theta)$ coincides with the sum of the homomorphisms
  $$
  H_q(\varphi|_{G_{\sigma}})\colon H_q(G_{\sigma})\to H_q(H_{f(\sigma)}).
  $$
  \item the map $H_*(\theta)\colon H_*^G(X)\to H_*^H(Y)$ sends $\CF_{p,q}(G,X)$ to $\CF_{p,q}(H,Y)$ for all~$p$ and~$q$, and $E^{\infty}_{p,q}(\theta)$ is the associate graded homomorphism.
 \end{itemize}
 \end{enumerate}

The functoriality, in particular, implies the following proposition.

\begin{propos}\label{propos_spec_seq_aut}
 Assume that $G\trianglelefteq \CG$ is a normal subgroup and $\CG$ acts cellularly and without rotations on a regular CW complex~$X$. Then the quotient group $\Delta=\CG/G$ acts by automorphisms of the spectral sequence $E^r_{p,q}=E^r_{p,q}(G,X)$ so that the following conditions are satisfied:
 \begin{enumerate}
  \item If $\delta\in\Delta$ and $h\in\CG$ is a preimage of~$\delta$, then $\delta$ acts on~$E^1_{p,q}$ via a direct sum of homomorphisms
 \begin{equation}
 \label{eq_composite_iota}
\begin{tikzcd}  [row sep=small]
\iota_{\sigma}\bigl(H_q(G_{\sigma})\bigr) \arrow[d, phantom, sloped, "\subseteq"]
\arrow[r,"{\iota_{\sigma}^{-1}}","\cong"'] &
H_q(G_{\sigma}) \arrow[r,"{h_*}","\cong"'] &
H_q(G_{h\sigma}) \arrow[r,"{\iota_{h\sigma}}","\cong"'] &
\iota_{h\sigma}\bigl(H_q(G_{h\sigma})\bigr)
\arrow[d, phantom, sloped, "\subseteq"]\\
{}\,E^1_{p,q} &&& {}\,E^1_{p,q}
\end{tikzcd}
\end{equation}
where $h_*$ is induced by the conjugation by~$h$.

  \item The filtration~\eqref{eq_filtration} is invariant under the natural action of~$\Delta$ on~$H_*^G(X)$. Moreover, the action of~$\Delta$ on~$E^{\infty}_{p,q}$ is the associated graded of the action on this filtration.
 \end{enumerate}
\end{propos}

\begin{remark}
 From the commutativity of~\eqref{eq_cd_iota} it follows that the homomorphisms~\eqref{eq_composite_iota} do not depend on the choice of the preimage~$h$ of~$\delta$ nor on the choice of representatives~$\sigma$ for $G$-orbits of $p$-cells.
\end{remark}

\begin{remark}
In fact, the assumptions that the group acts without rotations and that the CW complex~$X$ is regular are not necessary. However, without these assumptions, one would have to use homology with local coefficients and keep track of signs, which we would like to avoid.
\end{remark}

\section{Proof of Theorem~\ref{thm_annih}}\label{section_annih}
The assertion of Theorem~\ref{thm_annih} for $g\le 2$ is trivial, so we always assume that $g\ge 3$.

Consider a closed surface~$\Sigma_g$ of genus~$g$. The \textit{complex of curves}~$\CC(\Sigma_g)$ is a simplicial complex whose vertices are isotopy classes of essential simple closed curves on~$\Sigma_g$, and a tuple $\langle c_0,\ldots,c_k\rangle$ forms a simplex whenever the isotopy classes~$c_0,\ldots,c_k$ admit pairwise disjoint representatives.  Recall that a simple closed curve is called \textit{essential} if it is homotopically nontrivial. In what follows, we do not distinguish between a simple closed curve  and its isotopy class.

Choose a primitive homology class $x\in H_1(\Sigma_g;\Z)$. The \textit{complex of homologous curves}, defined by Putman~\cite{Put08} and denoted~$\CC_x(\Sigma_g)$, is the full subcomplex of~$\CC(\Sigma_g)$ spanned by all simple closed curves~$c$ that represent the homology class~$x$.  Minahan~\cite{Min23} proved that the complex~$\CC_x(\Sigma_g)$ is $(g-3)$-acyclic, that is, $\widetilde{H}_k\bigl(\CC_x(\Sigma_g);\Z\bigr)=0$ for all $k\le g-3$.

The Torelli group~$\I_g$ acts on~$\CC_x(\Sigma_g)$ cellularly and without rotations. In the cases of a surface with one puncture~$\Sigma_{g,1}$ or a surface with one boundary component~$\Sigma_g^1$, we define the actions of~$\I_{g,1}$ and~$\I_g^1$ on~$\CC_x(\Sigma_g)$ via the homomorphisms $\I_{g,1}\twoheadrightarrow \I_g$ and $\I_g^1\twoheadrightarrow \I_g$ obtained by filling in the puncture and attaching a disk to the boundary of~$\Sigma_g^1$, respectively.

Let us now prove Theorem~\ref{thm_annih}. Let $\I$ be any of the three Torelli groups~$\I_g$, $\I_{g,1}$, or~$\I_g^1$, and let $\Mod$ be the corresponding mapping class group (namely $\Mod_g$, $\Mod_{g,1}$, or~$\Mod_g^1$, respectively). Set $H=H_1(\Sigma_g;\Z)$. We want to prove that
\begin{equation}\label{eq_annih2}
 (t_x-1)^{k+1}H_k(\I;\Z)=0,\qquad k\le g-2,
\end{equation}
for all $x\in H$. Since $t_{nx}-1=t_x^{n^2}-1$ is divisible by~$t_x-1$, it suffices to prove~\eqref{eq_annih2} for only primitive classes~$x$.

Fix a primitive class $x\in H$ and consider the action of~$\I$ on the corresponding complex of homologous curves~$\CC_x(\Sigma_g)$. Let
\begin{equation*}
 E^1_{p,q}=\bigoplus_{\sigma\in\fX_p}H_q\left(\I_{\sigma};\Z\right) \ \Longrightarrow \ H_{p+q}^{\I}\bigl(\CC_x(\Sigma_g);\Z\bigr)
\end{equation*}
be the corresponding spectral sequence~\eqref{eq_spec_seq} and
\begin{equation}\label{eq_filtration2}
 0=\CF_{-1,k+1}\subseteq\CF_{0,k}\subseteq\CF_{1,k-1}\subseteq\cdots\subseteq\CF_{k,0}=H_k^{\I}\bigl(\CC_x(\Sigma_g);\Z\bigr)
\end{equation}
the corresponding filtration such that $E^{\infty}_{p,q}=\CF_{p,q}/\CF_{p-1,q+1}$.

Let $\CG$ denote the preimage of the infinite cyclic group~$\langle t_x\rangle$ under the surjective homomorphism $\varphi\colon\Mod\twoheadrightarrow \Sp(H)=\Sp_{2g}(\Z)$. Then we have a short exact sequence
$$
1\longrightarrow\I\longrightarrow\CG\stackrel{\varphi}{\longrightarrow}\langle t_x\rangle\longrightarrow 1.
$$
Since $t_x(x)=x$, the action of~$\I$ on~$\CC_x(\Sigma_g)$ extends to a cellular action of~$\CG$. Moreover, this action is without rotations, since for any simplex $\sigma=\langle c_0,\ldots,c_p\rangle$ and any connected component~$S$ of~$\Sigma_g\setminus(c_0\cup\cdots\cup c_p)$, the transvection~$t_x$ preserves the subgroup $H_1(S;\Z)\subseteq H$. By Proposition~\ref{propos_spec_seq_aut} the group $\langle t_x\rangle =\CG/\I$ acts by automorphisms of the spectral sequence~$E^r_{p,q}$.

\begin{propos}
 The action of\/~$t_x$ on the spectral sequence~$E^r_{p,q}$ is trivial.
\end{propos}

\begin{proof}
It suffices to prove that the action is trivial on the $E^1$ page. Fix a simple closed curve~$c$ that represents the homology class~$x$. Let $T_c$ be the right Dehn twist about~$c$. Then $T_c\in\CG$ and $\varphi(T_c)=t_x$. Since $\I$ acts transitively on simple closed curves in homology class~$x$, we can choose the set of representatives~$\fX_p$ for $\I$-orbits of $p$-simplices of~$\CC_x(\Sigma_g)$ so as to obtain that $c\in \sigma$ for all $\sigma\in\fX_p$. Then $T_c(\sigma)=\sigma$. Hence by Proposition~\ref{propos_spec_seq_aut} the element~$t_x$ acts on each direct summand $H_q(\I_{\sigma};\Z)$ of~$E^1_{p,q}$ via the automorphism induced by conjugation by~$T_c$. However, this automorphism is trivial, since $T_c$ commutes with all elements of the stabilizer~$\I_{\sigma}$.
\end{proof}

We now proceed with the proof of Theorem~\ref{thm_annih}. Since the action of~$t_x$ on the~$E^{\infty}$ page is the associated graded of the action of~$t_x$ on filtration~\eqref{eq_filtration2}, we obtain that
$$(t_x-1)\CF_{p,q}\subseteq\CF_{p-1,q+1}$$
for all~$p$ and~$q$. Therefore,
$$
(t_x-1)^{k+1}H_k^{\I}(\CC_x(\Sigma_g);\Z)=0\qquad \text{for all }k.
$$
Finally, since by the above mentioned result of Minahan the complex~$\CC_x(\Sigma_g)$ is $(g-3)$-acyclic, we see that the natural homomorphism
$$
H_k^{\I}\bigl(\CC_x(\Sigma_g);\Z\bigr)\to H_k\bigl(\I;\Z\bigr)
$$
is an isomorphism for $k\le g-3$ and a surjection for $k=g-2$. Consequently,
$$
(t_x-1)^{k+1}H_k\bigl(\I;\Z\bigr)=0\qquad \text{for }k\le g-2.
$$
This completes the proof of Theorem~\ref{thm_annih}.

\section{Universal $\CT$-unipotent modules and groups of finite width}
\label{section_U}

Let $\Gamma$ be a group and $\CT\subseteq\Gamma$ a conjugation-invariant subset.
The following is a generalization of Definition~\ref{defin_transv_unip}.

 \begin{defin}\label{defin_T_unip}
 A $\Gamma$-module~$M$ will be called \textit{$\CT$-unipotent} if there exists a positive integer~$n$ such that $(\tau-1)^nM=0$ for all $\tau\in\CT$. The smallest $n$ with this property will be called the \textit{unipotency index} of~$M$.
\end{defin}

We need the following construction of a \textit{universal $\CT$-unipotent module} of unipotency index~$n$. Fix a number~$n$ and consider the following two-sided ideal of the group ring~$\Z[\Gamma]$:
$$
J_n=\sum_{\tau\in\CT}(\tau-1)^{n}\Z[\Gamma]=\sum_{\tau\in\CT}\Z[\Gamma](\tau-1)^{n}.
$$
We set
\begin{equation}\label{eq_Un}
\CU_n=\Z[\Gamma]/J_n.
\end{equation}
Multiplication from the left by elements of~$\Gamma$ endows~$\CU_n$ with the structure of a left $\CT$-unipotent $\Gamma$-module. (Note that $\CU_n$ has also a structure of a right $\Gamma$-module but this structure will not be interesting to us.)

The module $\CU_n$ will be referred to as the \textit{universal $\CT$-unipotent module} of unipotency index~$n$. The image of the element $1\in\Z[\Gamma]$ in~$\CU_n$ will be denoted by~$u$ and called the \textit{fundamental generator} of~$\CU_n$. Of course, the ideal~$J_n$ and the module~$\CU_n$ depend on the set~$\CT$, but we prefer to suppress this dependence from the notation. The obvious universality property for~$\CU_n$ is as follows.

\begin{propos}\label{propos_universal}
 Let $M$ be a $\CT$-unipotent $\Gamma$-module of unipotency index~$\le n$. Then for each $z\in M$ there is a unique  homomorphism of (left) $\Gamma$-modules $\psi_z\colon \CU_n\to M$ with $\psi_z(u)=z$. Moreover, the map $z\mapsto\psi_z$ gives a well-defined isomorphism of abelian groups $$\Psi\colon M\stackrel{\cong}{\longrightarrow}\Hom_{\Z[\Gamma]}(\CU_n,M).$$
\end{propos}

Now, assume that $\Gamma$ is a group of finite width, see Definition~\ref{defin_fw}.
We need the following easy proposition.

\begin{propos}\label{propos_U_fg}
 Let $\Gamma$ be a group of finite width, $\Gamma=\langle\gamma_1\rangle\cdots \langle\gamma_m\rangle$, let $\CT\subseteq\Gamma$ be a conjugacy-invariant  subset that contains all elements $\gamma_1,\ldots,\gamma_m$, and let\/ $\CU_n$ be the corresponding universal $\CT$-unipotent $\Gamma$-module of unipotency index~$n$. Then\/ $\CU_n$ is a finitely generated abelian group. More precisely, $\CU_n$ is generated as an abelian group by the $n^m$ elements $\gamma_1^{l_1}\cdots\gamma_m^{l_m}u$ with $0\le l_i<n$.
\end{propos}

\begin{proof}
For each $\tau\in\CT$ and each element $v\in \CU_n$, we have that $(\tau-1)^nv=0$. It follows easily that any element $\tau^kv$ with $k\in\Z$ can be expressed as a $\Z$-linear combination of the elements $v,\tau v,\ldots,\tau^{n-1}v$.

Each element of~$\Gamma$ can be written as $\gamma=\gamma_1^{k_1}\cdots\gamma_m^{k_m}$.
Taking $\tau=\gamma_1$ and $v=\gamma_2^{k_2}\cdots\gamma_m^{k_m}u$, we can write $\gamma u$ as a $\Z$-linear combination of elements of the form $\gamma_1^{l_1}\gamma_2^{k_2}\cdots\gamma_m^{k_m}u$ with $0\le l_1<n$. Then, using
$$
\gamma_1^{l_1}\gamma_2^{k_2}\cdots\gamma_m^{k_m}u=\left(\gamma_1^{l_1}\gamma_2\gamma_1^{-l_1}\right)^{k_2}\gamma_1^{l_1}\gamma_3^{k_3}\cdots\gamma_m^{k_m}u
$$
and setting $\tau=\gamma_1^{l_1}\gamma_2\gamma_1^{-l_1}$ and $v=\gamma_1^{l_1}\gamma_3^{k_3}\cdots\gamma_m^{k_m}u$, we can further write $\gamma u$ as a $\Z$-linear combination of elements of the form $\gamma_1^{l_1}\gamma_2^{l_2}\gamma_3^{k_3}\cdots\gamma_m^{k_m}u$ with $0\le l_1,l_2<n$. Proceeding inductively, we eventually express $\gamma u$ as a $\Z$-linear combination of elements $\gamma_1^{l_1}\cdots\gamma_m^{l_m}u$ where $0\le l_i<n$ for each~$i$. This proves the proposition.
\end{proof}

\section{Bounded generation of~$\Sp_{2g}(\Z)$ and proof of Theorem~\ref{thm_main} for $k=2$}
\label{section_bounded}

In this section we first explain how the following proposition can be extracted from the paper~\cite{Tav91} by Tavgen, and then prove Theorem~\ref{thm_main} for $k=2$. For all definitions and facts concerning Chevalley groups, see~\cite{Ste67}.

\begin{propos}[= Proposition~\ref{propos_fw}]\label{propos_Tav}
 If $g\ge 2$, then there exists a finite sequence of (not necessarily distinct) elements $x_1,\ldots,x_m\in\Z^{2g}$ such that the group~$\Sp_{2g}(\Z)$ has finite width with respect to the corresponding sequence of symplectic transvections: $\Sp_{2g}(\Z)=\langle t_{x_1}\rangle\cdots \langle t_{x_m}\rangle$.
\end{propos}

A group $\Gamma$ is said to be \textit{boundedly generated} by a subset~$A\subseteq \Gamma$ if there exists a constant~$K$ such that every element $\gamma\in \Gamma$ can be written as $\gamma=a_1\cdots a_k$ with $a_i\in A$ and $k\le K$. Carter and Keller~\cite{CaKe83} proved that the group~$\SL_n(\Z)$ is boundedly generated by elementary matrices, provided that $n\ge 3$. Tavgen~\cite[Theorem~A]{Tav91} then generalized this result to all Chevalley groups~$G(\Phi,\Z)$, where $\Phi$ is a reduced irreducible root system of rank~$\ge 2$. Namely, he proved that $G(\Phi,\Z)$ is boundedly generated by the set of all \textit{root unipotents} (also called \textit{root elements}) $t_{\alpha}^{\xi}=\exp(\xi X_{\alpha})$, where $\alpha\in\Phi$,  $\xi\in\Z$, and $X_{\alpha}$ is the element  corresponding to~$\alpha$ of the Chevalley basis for the Lie algebra~$\mathfrak{g}(\Phi)$. In particular, this result holds for all groups $\Sp_{2g}(\Z)=G(C_g,\Z)$ with $g\ge 2$. (A somewhat  weaker result for~$\Sp_{2g}(\Z)$ with $g\ge 3$ and a larger set of generators was earlier obtained by Zakiryanov~\cite{Zak85}.)

Since $\Phi$ is a finite set, Tavgen's result immediately implies that
$\Sp_{2g}(\Z)=\langle t_{\alpha_1}\rangle\cdots \langle t_{\alpha_m}\rangle$ for some finite sequence of (not necessarily distinct) root unipotents $t_{\alpha_i}=\exp(X_{\alpha_i})$. Indeed, one can take for $\alpha_1,\ldots,\alpha_m$ any finite sequence of roots that contain as subsequences all possible sequences of roots of length~$K$. For the following standard facts, see, for instance,~\cite[Sect.~14]{StVa00}:
\begin{itemize}
 \item If $\alpha$ is a long root, then $t_{\alpha}$ is a symplectic transvection~$t_x$ for some primitive $x\in\Z^{2g}$.
 \item If $\alpha$ is a short root, then~$t_{\alpha}$ is a so-called \textit{Eichler--Siegel--Dickson transvection}, i.\,e., a transformation of the form
 $$
 t_{x,y}z=z+(x,z)y+(y,z)x,
 $$
 where $x$ and~$y$ are primitive vectors in~$\Z^{2g}$ with $(x,y)=0$. But
 $$
 t_{x,y}=t_{x+y}t_x^{-1}t_y^{-1}
 $$
 and the symplectic transvections~$t_{x+y}$, $t_x$, and~$t_y$ pairwise commute. Hence,
 $$
 \langle t_{\alpha}\rangle \subseteq \langle t_{x+y}\rangle\langle t_{x}\rangle\langle t_{y}\rangle.
 $$
 \end{itemize}
Proposition~\ref{propos_Tav} follows.

\begin{proof}[Proof of Theorem~\ref{thm_main} for $k=2$]
Set~$\Gamma=\Sp_{2g}(\Z)$, and let $\CT\subseteq\Sp_{2g}(\Z)$ be the subset consisting of all symplectic transvections~$t_x$. Consider the universal $\CT$-unipotent $\Sp_{2g}(\Z)$-module~$\CU_3$ of unipotency index~$3$. We have $g\ge k+2=4$. So by Proposition~\ref{propos_Tav} the group $\Sp_{2g}(\Z)$ has finite width with respect to a sequence of elements from~$\CT$. Then by Proposition~\ref{propos_U_fg} we obtain that $\CU_3$ is a finitely generated abelian group. By Theorem~\ref{thm_annih} the $\Sp_{2g}(\Z)$-module $H_2(\I_{g,p}^b;\Z)$ is $\CT$-unipotent of unipotency index~$\le3$. Hence by Proposition~\ref{propos_universal} for each homology class $z\in H_2(\I_{g,p}^b;\Z)$ there is a homomorphism of $\Sp_{2g}(\Z)$-modules $\psi_z\colon\CU_3\to H_2(\I_{g,p}^b;\Z)$ with $\psi_z(u)=z$. Now, by a theorem of Kassabov and Putman~\cite{KaPu20} the group~$H_2(\I_{g,p}^b;\Z)$ is generated as an $\Sp_{2g}(\Z)$-module by a finite number of elements $z_1,\ldots,z_s$. It follows that $H_2(\I_{g,p}^b;\Z)$ is generated as an abelian group by the images of the homomorphisms $\psi_{z_1},\ldots,\psi_{z_s}$. Consequently, it is a finitely generated abelian group.
\end{proof}

\section{Finite generation of the coinvariants}\label{section_coinv}

The aim of the present section is to prove the following proposition. We fix $g$, $b$, and~$p$ with $b+p\le 1$.

\begin{propos}\label{propos_fg_coinv}
 Assume that the groups $H_i\bigl(\I_{g,p}^b;\Z\bigr)$ are finitely generated for $i<k$. Then for any finite-index subgroup $\Lambda\le\Sp_{2g}(\Z)$ and any $\Lambda$-module~$A$ that is finitely generated as an abelian group, the coinvariant group $\left(H_k\bigl(\I_{g,p}^b;\Z\bigr)\otimes A\right)_{\Lambda}$ is finitely generated.
\end{propos}

Recall that a group~$G$ is of \textit{type~FL} if $\Z$ has a finite free resolution over~$\Z [G]$, and $G$ is of \textit{type~VFL} if $G$ has a finite-index subgroup of type~FL. If $G$ is of type~VFL and $A$ is a $G$-module that is finitely generated as an abelian group, then $H_k(G;A)$ is finitely generated for every~$k$. Moreover, if $G$ is a group of type~VFL, then any finite-index subgroup of~$G$ is also of type~VFL. These and other basic facts on groups of types~FL and~VFL can be found in~\cite{Ser71}.

We will need the following two classical results.

\begin{theorem}[Raghunathan,~\cite{Rag68}]\label{thm_Rag}
Let $\Gamma$ be an arithmetic subgroup of a connected semi-simple algebraic group defined over~$\Q$. Then $\Gamma$ is of type~VFL.
\end{theorem}

\begin{theorem}[Ivanov,~\protect{\cite[Sect.~6.2]{Iva87}}]\label{thm_Iva}
The mapping class groups~$\Mod_{g,p}^b$ are of type~VFL.
\end{theorem}

\begin{proof}[Proof of Proposition~\ref{propos_fg_coinv}]
 Let $\I$ be any of the three Torelli groups~$\I_g$, $\I_{g,1}$, or~$\I_g^1$, and let $\Mod$ be the corresponding mapping class group (namely $\Mod_g$, $\Mod_{g,1}$, or~$\Mod_g^1$, respectively).

Let $\CG$ be the preimage of~$\Lambda$ under the surjection $$\Mod\twoheadrightarrow\Sp_{2g}(\Z).$$ We can consider~$A$ as a $\CG$-module, with the Torelli subgroup~$\I$ acting trivially on~$A$.

First, let us prove that the coinvariant group $H_k(\I;A)_{\Lambda}$ is finitely generated. Here the structure of a $\Lambda$-module on~$H_k(\I;A\bigr)$ is as follows. The group~$\CG$ acts by conjugations on~$\I$ and acts on~$A$. So $\CG$ acts on~$H_k(\I;A)$. But the action of~$\I$ on the homology of itself is trivial, so $H_k\bigl(\I;A\bigr)$ becomes a $\Lambda$-module.

Consider the short exact sequence
$$
1\longrightarrow \I\longrightarrow\CG\longrightarrow\Lambda\longrightarrow 1
$$
and the corresponding Lyndon--Hochschild--Serre spectral sequence
$$
E^2_{p,q}=H_p\left(\Lambda,H_q(\I;A)\right)\ \Longrightarrow \  H_{p+q}(\CG;A).
$$

By the assumption made we have that the groups $H_q(\I;\Z)$ are finitely generated for $q<k$. By the universal coefficients theorem the groups $H_q(\I;A)$ are also finitely generated for $q<k$. (Note that the coefficients are untwisted, since $\I$ acts trivially on~$A$.) Using Theorem~\ref{thm_Rag}, we obtain that all groups~$E^2_{p,q}$ with $q<k$ are finitely generated.

From Theorem~\ref{thm_Iva} it follows that the group~$H_k(\CG;A)$ is finitely generated. Since $E^{\infty}_{0,k}$ is a subgroup of~$H_k(\CG;A)$, we see that $E^{\infty}_{0,k}$ is finitely generated. The group~$E^{\infty}_{0,k}$ is obtained from the group $E^2_{0,k}=H_k(\I;A)_{\Lambda}$ by taking consecutive quotients by the images of the differentials $d^2,\ldots,d^{k+1}$. But we have already proved that the groups~$E^r_{r,k-r+1}$ that are the sources of the differentials  $d^r$ landing in~$E^r_{0,k}$ are finitely generated (where $r=2,\ldots,k+1$). Therefore, the group~$H_k(\I;A)_{\Lambda}$ is finitely generated.

Second, the universal coefficients theorem provides the short exact sequence
\begin{equation}\label{eq_uct}
0\longrightarrow H_k(\I;\Z)\otimes A\longrightarrow H_k(\I;A)\longrightarrow\Tor\left(H_{k-1}(\I;\Z),A\right)\longrightarrow 0,
\end{equation}
where $\Tor=\Tor^{\Z}_1$. Moreover, by the assumption made the group~$H_{k-1}(\I;\Z)$ is finitely generated, so the group
$$
T=\Tor\left(H_{k-1}(\I;\Z),A\right)
$$
is also finitely generated. The group~$\Lambda$ acts by automorphisms of exact sequence~\eqref{eq_uct}. The corresponding long exact sequence in homology reads as
$$
\ldots \longrightarrow H_1(\Lambda; T) \longrightarrow \left(H_k(\I;\Z)\otimes A\right)_{\Lambda}
\longrightarrow H_k(\I;A)_{\Lambda}
\longrightarrow T_{\Lambda}\longrightarrow 0.
$$
From Theorem~\ref{thm_Rag} it follows that the group~$H_1(\Lambda; T)$ is finitely generated. We have already proved that the group $H_k(\I;A)_{\Lambda}$ is finitely generated. Thus, the group $\left(H_k(\I;\Z)\otimes A\right)_{\Lambda}$ is also finitely generated.
\end{proof}

\section{Quasiprojection~$q$}\label{section_quasiproj}

This section contains a representation-theoretic construction, which is then  used in Section~\ref{section_proof_main} to prove Theorem~\ref{thm_alg}. This construction will be needed for us in the case where $\Gamma$ is an arithmetic group. However, the construction itself does not rely on arithmeticity, so in this section~$\Gamma$ is an arbitrary group.

For a group $\Gamma$ and a commutative ring~$R$, we denote by $\varepsilon\colon R[\Gamma]\to R$ the augmentation homomorphism. If $V$ is an $R[\Gamma]$-module and $q\in R[\Gamma]$, then, with some abuse of notation, we denote by~$q$ also the map $V\to V$ given by $v\mapsto qv$.

\begin{lem}\label{lem_q}
Let $\Gamma$ be a group and let $V$ be a $\Z[\Gamma]$-module that is finitely generated as an abelian group. Assume that the $\Q[\Gamma]$-module $V\otimes\Q$ is semisimple. Then there exists an element $q\in\Z[\Gamma]$ such that
\begin{enumerate}
 \item $\varepsilon(q)\ne 0$,
 \item $qV\subseteq V^{\Gamma}$, i.\,e., $\gamma qv=qv$ for all $\gamma\in \Gamma$ and~$v\in V$,
 \item the homomorphism $q\colon V\to V^{\Gamma}$ is a $\Gamma$-invariant map, i.\,e., $q\gamma v=qv$ for all $\gamma\in \Gamma$ and~$v\in V$.
\end{enumerate}
\end{lem}

\begin{remark}
 If $\varepsilon(q)=1$, then $q$ would be a $\Gamma$-invariant projection $V\twoheadrightarrow V^{\Gamma}$. However, such a projection may not exist. The element $q$ will be called a \textit{quasiprojection}.
\end{remark}

To prove Lemma~\ref{lem_q} we need the following auxiliary lemma.

\begin{lem}\label{lem_q_step}
 Let $W$ be an irreducible $\Q[\Gamma]$-module such that $W$ is non-trivial, i.\,e., not isomorphic to~$\Q$ with the trivial $\Gamma$-action. Let $w\in W$ be a nonzero vector. Then there exists an element $r\in\Z[\Gamma]$ such that $\varepsilon(r)\ne0$ and  $rw=0$.
\end{lem}

\begin{proof}
Since $W$ is irreducible, we see that $\Q[\Gamma]w=W$. Consider the annihilator
$$
\Ann(w)=\left\{r\in\Q[\Gamma]: rw=0\right\}.
$$
Then $\Ann(w)$ is a left ideal of~$\Q[\Gamma]$ and $W\cong\Q[\Gamma]/\Ann(w)$. If $\Ann(w)$ were contained in~$\ker\varepsilon$, then $\varepsilon$ would induce a surjective homomorphism of $\Q[\Gamma]$-modules $W\twoheadrightarrow\Q$. However, such a homomorphism cannot exist, since $W$ is irreducible and not isomorphic to~$\Q$ with the trivial $\Gamma$-action. Hence, there exists an element $r\in\Ann(w)$ with $\varepsilon(r)\ne 0$. Multiplying $r$ by a nonzero integer, we can obtain a required element from~$\Z[\Gamma]$.
\end{proof}

\begin{proof}[Proof of Lemma~\ref{lem_q}]
We start with the case when $V$ is itself a semisimple finite-dimensional $\Q[\Gamma]$-module. Though $V$ is not finitely generated as an abelian group, let us prove that for such~$V$ there exists a quasiprojection~$q$ with the required properties~(1)--(3).

First, suppose that $V$ is an  irreducible $\Q[\Gamma]$-module. If $V\cong\Q$ is a trivial $\Q[\Gamma]$-module, then we can take $q=1$. So we assume that $V$ is nontrivial; then $V^{\Gamma}=0$. Take a nonzero vector $v_1\in V$. By Lemma~\ref{lem_q_step} there exists an element $r_1\in\Z[\Gamma]$ such that  $\varepsilon(r_1)\ne 0$ and $r_1v_1=0$. If $r_1V=0$, then we are done. If $r_1V\ne0$, then we take a nonzero vector $v_2\in r_1V$ and choose an element $r_2\in\Z[\Gamma]$ such that $\varepsilon(r_2)\ne 0$ and $r_2v_2=0$. We then proceed in the same way: at the $i$th step we take a nonzero vector $v_i\in r_{i-1}\cdots r_1V$ and choose an element $r_i\in\Z[\Gamma]$ such that $\varepsilon(r_i)\ne 0$ and $r_iv_i=0$. The dimensions satisfy the strict inequalities
$$
\dim V>\dim(r_1V)>\dim(r_2r_1V)>\ldots,
$$
so the process terminates after finitely many steps, and we obtain an element $q=r_k\cdots r_1$ satisfying $\varepsilon(q)\ne 0$ and~$qV=0$.

Second, consider an arbitrary semisimple finite-dimensional $\Q[\Gamma]$-module~$V$. We write
$$
V=V^{\Gamma}\oplus V_1\oplus \cdots\oplus V_m,
$$
where $V_1,\ldots,V_m$ are nontrivial irreducible representations. By the above there exist elements $q_1,\ldots,q_m\in\Z[G]$ such that $\varepsilon(q_i)\ne 0$ and $q_iV_i=0$ for each~$i$. Put  $q=q_1\cdots q_m$. Then $\varepsilon(q)\ne 0$ and $q$ annihilates every of the summands $V_1,\ldots,V_m$. Moreover, $qv=\varepsilon(q)v$ for all $v\in V^{\Gamma}$. It follows that $qV\subseteq V^{\Gamma}$ and the map $q\colon V\to V^{\Gamma}$ is $\Gamma$-invariant.

Finally, consider the general case of an arbitrary $\Z[\Gamma]$-module~$V$ that is finitely generated as an abelian group and such that  $V\otimes\Q$ is semisimple. By the above, we already know the existence of the required quasiprojection~$q_{\Q}\in\Z[\Gamma]$ for the $\Q[\Gamma]$-module~$V\otimes\Q$. Since $V$ is a finitely generated abelian group, its torsion subgroup $T=\tors(V)$ is finite. Let $a$ be the largest order of an element in~$T$. Then $aT=0$. It follows that $q=aq_{\Q}$ can serve as the required quasiprojection for~$V$.
\end{proof}

\section{Preliminaries on arithmetic groups}\label{section_arithm}

As in the Introduction, let $\bG$ be a connected semisimple algebraic subgroup of~$\bGL_n$ defined over~$\Q$ and let $\Gamma$ be a Zariski dense arithmetic subgroup of~$\bG(\Q)$. We shall restrict ourselves to groups that have the following property.
\smallskip

\textbf{Almost Algebraic Representations (AAR) Property:} \textit{Every representation of~$\Gamma$ on a finite-dimensional $\Q$-vector space is almost algebraic.}
\smallskip

The only example of interest to us will be the groups~$\Sp_{2g}(\Z)$ with $g\ge 2$. In this particular case, the AAR property was established by Bass, Milnor, and Serre~\cite[Theorem~16.2]{BMS67}, see  also~\cite{Rag67}. Note, however, that in fact this property is known for a wide class of arithmetic groups, including all arithmetic subgroups of simple algebraic groups of $\Q$-rank~$\ge2$ defined over~$\Q$ (see~\cite{Ser79}; in this case the property follows from Margulis Superrigidity Theorem), as well as all groups possessing a certain form of the congruence subgroup property (see~\cite{Lub80}).

We need the following well known consequences of the AAR property. In the next two corollaries, $\Gamma$ is a Zariski dense arithmetic subgroup of a connected semisimple algebraic group~$\bG(\Q)$ defined over~$\Q$ such that that $\Gamma$ possesses the AAR property.

\begin{cor}\label{cor_complete_red}
 Every representation of\/~$\Gamma$ on a finite-dimensional $\Q$-vector space is semisimple.
\end{cor}

\begin{cor}\label{cor_decompose_prod}
 Let $\rho\colon\Gamma\to\GL(V)$ be a representation of~$\Gamma$ on a finite-dimensional $\Q$-vector space. Then $\rho$ can be written as $\rho(\gamma)=f(\gamma)u(\gamma)$, where
 \begin{itemize}
  \item $f\colon \Gamma\to\GL(V)$ is an algebraic representation,
  \item $u\colon \Gamma\to\GL(V)$ is a representation that factors through a finite quotient group~$\Gamma/\Lambda$,
  \item $f(\gamma_1)$ commutes with $u(\gamma_2)$ for all $\gamma_1,\gamma_2\in\Gamma$.
 \end{itemize}
 \end{cor}

Corollary~\ref{cor_complete_red} is deduced from the AAR property in~\cite[Corollary~16.4]{BMS67}. Corollary~\ref{cor_decompose_prod} follows from a computation made by Serre~\cite[p.~502]{Ser70} in the case of~$\bSL_2$; for an explanation, why this works in the general case, see Proposition~5.1 and its proof in~\cite{Lub80}.

\section{Proof of Theorem~\ref{thm_main}}\label{section_proof_main}

In this section we formulate an algebraic theorem (Theorem~\ref{thm_alg}) that will enable us to prove finite generation over~$\Z$ of a transvection unipotent $\Sp_{2g}(\Z)$-module~$M$, relying not on the finite generation of~$M$ over~$\Z\bigl[\Sp_{2g}(\Z)\bigr]$ but on a much weaker condition---the finite generation of suitable coinvariant groups~$(M\otimes A)_{\Lambda}$, where $\Lambda\le\Sp_{2g}(\Z)$ is a subgroup of finite index. This theorem may be of further interest not only for the groups~$\Sp_{2g}(\Z)$ but also for the groups~$\SL_n(\Z)$, in view of possible applications to automorphisms of free groups. We therefore state it in some reasonable generality.

\begin{theorem}\label{thm_alg}
 Let\/ $\Gamma$ be a Zariski dense arithmetic subgroup of a semisimple algebraic group~$\bG(\Q)$ defined over~$\Q$ such that $\Gamma$ possesses the AAR property. Let $\CT\subseteq\Gamma$ be a conjugation-invariant subset. Suppose that  $\Gamma=\langle\gamma_1\rangle\cdots\langle\gamma_m\rangle$ with $\gamma_i\in\CT$. Assume that a $\CT$-unipotent $\Gamma$-module~$M$ satisfies the following condition:
 \begin{itemize}
  \item[$(*)$] for any finite-index subgroup  $\Lambda\le\Gamma$ and any $\Lambda$-module~$A$ that is finitely generated as an abelian group, the coinvariant group $(M\otimes A)_{\Lambda}$ is finitely generated.
 \end{itemize}
 Then $M$ is finitely generated as an abelian group.
\end{theorem}

Before proving this theorem, we deduce Theorem~\ref{thm_main} from it.

\begin{proof}[Proof of Theorem~\ref{thm_main}]
 If $g\le 2$, then there is nothing to prove.
 For each~$g\ge 3$, we proceed by induction on~$k$. The induction base $k=0$ is trivial. We assume that the groups~$H_i(\I_{g,p}^b;\Z)$ are finitely generated for $i<k$, and prove that the group $H_k(\I_{g,p}^b;\Z)$ is also finitely generated, provided that $k\le g-2$. Take for $\CT$ the set of all symplectic transvections $t_x\in\Sp_{2g}(\Z)$, where $x\in H_1(\I_{g,p}^b;\Z)=\Z^{2g}$. By Proposition~\ref{propos_Tav} the group $\Sp_{2g}(\Z)$ has finite width with respect to a sequence of symplectic transvections. The $\Sp_{2g}(\Z)$-module $H_k(\I_{g,p}^b;\Z)$ is $\CT$-unipotent of unipotency index~$\le k+1$ by Theorem~\ref{thm_annih} and satisfies condition~$(*)$ by Proposition~\ref{propos_fg_coinv}. Thus, by Theorem~\ref{thm_alg} we conclude that $H_k(\I_{g,p}^b;\Z)$ is a finitely generated abelian group.
\end{proof}

In the remainder of this section, we prove Theorem~\ref{thm_alg}. To do this we first need to study some properties of the universal $\CT$-unipotent modules from Section~\ref{section_U}. We assume that $\Gamma$ and~$\CT$ are as in Theorem~\ref{thm_alg}.

Fix a positive integer~$n$. Let~$\CU_n$ be the universal $\CT$-unipotent $\Gamma$-module of unipotency index~$n$ given by~\eqref{eq_Un}. By Proposition~\ref{propos_U_fg} the module~$\CU_n$ is a finitely generated abelian group. We denote by~$\Tors(\CU_n)$ the torsion subgroup of~$\CU_n$, and set
\begin{gather*}
\CU_n'=\CU_n/\Tors(\CU_n),\\
\CV_n=\Hom_{\Z}(\CU'_n,\CU_n).
\end{gather*}
(As usually, if $M$ and~$N$ are $\Gamma$-modules, then the structure of a $\Gamma$-module on $\Hom_{\Z}(M,N)$ is defined by $(\gamma\xi)(z)=\gamma\cdot\xi(\gamma^{-1}z)$, where $\gamma\in\Gamma$, $\xi\in\Hom_{\Z}(M,N)$, and $z\in M$.)

Since $\CU_n$ is a finitely generated abelian group, we have that $\CV_n$ is also a finitely generated abelian group. Moreover, by Corollary~\ref{cor_complete_red} the $\Q[\Gamma]$-module $\CV_n\otimes\Q$ is semisimple. Hence by Lemma~\ref{lem_q} there exists a quasiprojection $q_n\in \Z[\Gamma]$ such that $\varepsilon(q_n)\ne 0$, $q_n\CV_n\subseteq\CV_n^{\Gamma}$ and the map $q_n\colon \CV_n\to\CV_n^{\Gamma}$ is $\Gamma$-invariant. We fix such a quasiprojection.

\begin{propos}\label{propos_aux_main_new}
 Let $M$ be a $\CT$-unipotent $\Gamma$-module of unipotency index~$\le n$. Then
  $$
  q_n\Hom_{\Z}(\CU'_n,M)\subseteq\Hom_{\Z}(\CU'_n,M)^{\Gamma}=\Hom_{\Z[\Gamma]}(\CU'_n,M),
  $$
  and the map
  $$
  q_n\colon \Hom_{\Z}(\CU'_n,M)\to\Hom_{\Z[\Gamma]}(\CU'_n,M)
  $$
  is $\Gamma$-invariant.
\end{propos}

\begin{proof}
 By Proposition~\ref{propos_universal} for each element $z\in M$ there is a unique $\Gamma$-module homomorphism $\psi_z\colon \CU_n\to M$ such that $\psi_z(u)=z$. Postcomposition with $\psi_z$ defines a $\Gamma$-module homomorphism
 $$
 \psi_{z*}\colon \CV_n\to \Hom_{\Z}(\CU'_n,M).
 $$
 We have a commutative diagram
 $$
\begin{tikzcd}
 \CV_n \arrow[r,"q_n"] \arrow[d,"{\psi_{z*}}"] &
 \CV_n \arrow[d,"{\psi_{z*}}"]\\
 \Hom_{\Z}(\CU'_n,M) \arrow[r,"q_n"]  &
 \Hom_{\Z}(\CU'_n,M)
\end{tikzcd}
 $$
 By the construction of~$q_n$ we know that the upper horizontal arrow is a $\Gamma$-invariant map whose image is contained in the invariant submodule~$\CV_n^{\Gamma}$. Let us prove that the lower horizontal arrow is also a $\Gamma$-invariant map whose image is contained in the invariant submodule $\Hom_{\Z}(\CU_n',M)^{\Gamma}$. To do this, it suffices to prove that the union of the submodules $\psi_{z*}(\CV_n)$ over all $z\in M$ generates $\Hom_{\Z}(\CU'_n,M)$ as a an abelian group. We now show that this is indeed the case. Let $e_1,\ldots,e_d$ be a basis of the free abelian group~$\CU'_n$. For each $i=1,\ldots,d$, consider a $\Z$-linear map $\sigma_i\colon \CU'_n\to \CU_n,$ such that $\sigma_i(e_i)=u$ and $\sigma_i(e_j)=0$ for $i\ne j$. Then for any $\xi\in \Hom_{\Z}(\CU'_n,M)$, we have that
 $$
 \xi=\sum_{i=1}^d\psi_{\xi(e_i)}\circ\sigma_i =\sum_{i=1}^d\psi_{\xi(e_i)*}(\sigma_i),
 $$
 as required.
\end{proof}

\begin{cor}\label{cor_qn_image}
 If $M$ satisfies condition~$(*)$, then the group~$q_n\Hom_{\Z}(\CU'_n,M)$ is finitely generated.
\end{cor}

\begin{proof}
 Since the homomorphism
  $$
  q_n\colon \Hom_{\Z}(\CU'_n,M)\to\Hom_{\Z[\Gamma]}(\CU'_n,M)
  $$
  is a $\Gamma$-invariant map, it factors through the coinvariant group
  $$
  \Hom_{\Z}(\CU'_n,M)_{\Gamma}\cong\bigl(M\otimes\Hom_{\Z}(\CU'_n,\Z)\bigr)_{\Gamma},
  $$
  which is finitely generated by condition~$(*)$.
\end{proof}

Let $a_n$ be the largest order of an element of~$\Tors(\CU_n)$. In other words, $a_n$ is the smallest positive integer such that $a_n\Tors(\CU_n)=0$. Set
$$b_n=|\varepsilon(q_n)|,\qquad c_n=a_nb_n.$$

\begin{propos}\label{propos_qn_cn}
For any $\CT$-unipotent $\Gamma$-module~$M$ of unipotency index~$\le n$, the group $q_n\Hom_{\Z}(\CU'_n,M)$ contains a subgroup that is isomorphic as an abelian group to the submodule~$c_nM\subseteq M$.
\end{propos}

\begin{proof}
 Consider the commutative diagram:
 $$
 \begin{tikzcd}
  M \arrow[r,"\Psi"',"\cong"] & \Hom_{\Z[\Gamma]}(\CU_n,M) & \\
  a_nM \arrow [u, hookrightarrow] \arrow[d,twoheadrightarrow, "{\varepsilon(q_n)}"]
  \arrow [r,hookrightarrow,"{\Psi}"'] &
  \Hom_{\Z[\Gamma]}(\CU'_n,M) \arrow [u, hookrightarrow] \arrow [r,hookrightarrow]
  \arrow[d,twoheadrightarrow,"{\varepsilon(q_n)}"] &
  \Hom_{\Z}(\CU'_n,M) \arrow[d,twoheadrightarrow,"{q_n}"] \\
  c_nM \arrow [r,hookrightarrow,"{\Psi}"'] & b_n\Hom_{\Z[\Gamma]}(\CU'_n,M)
  \arrow [r,hookrightarrow]
  & q_n\Hom_{\Z}(\CU'_n,M)
 \end{tikzcd}
 $$
 Here the upper horizontal arrow is the isomorphism~$\Psi$ from Proposition~\ref{propos_universal}. The middle and lower arrows are the restrictions of this isomorphism to the subgroups~$a_nM$ and~$c_nM$, respectively. The inclusion $\Hom_{\Z[\Gamma]}(\CU'_n,M)\hookrightarrow \Hom_{\Z[\Gamma]}(\CU_n,M)$ is induced by the surjection $\CU_n\twoheadrightarrow \CU_n'$. The group $\Psi(a_nM)$ is contained in the subgroup $\Hom_{\Z[\Gamma]}(\CU'_n,M)$ because for any homomorphism $\xi\colon \CU_n\to M$, the homomorphism~$a_n\xi$ vanishes on~$\Tors(\CU_n)$. The commutativity of the lower right square follows from Proposition~\ref{propos_aux_main_new}. From the lower row of the diagram we see that $c_nM$ injects into $q_n\Hom_{\Z}(\CU'_n,M)$.
\end{proof}

\begin{proof}[Proof of Theorem~\ref{thm_alg}]
From Corollary~\ref{cor_qn_image} and Proposition~\ref{propos_qn_cn} it follows that $c_nM$ is a finitely generated abelian group. So it suffices to prove that the quotient group~$M/c_nM$ is finite.

Since $\CU_n$ is a finitely generated abelian group, we see that $\CU_n/c_n\,\CU_n$ is a finite abelian group. Let $\Lambda$ be the kernel of the homomorphism
 $$
 \Gamma\to \Aut(\CU_n/c_n\,\CU_n).
 $$
Then $\Lambda$ has finite index in~$\Gamma$.

 Let us show that $\Lambda$ acts trivially on~$M/c_nM$. Indeed, suppose that $y\in M/c_nM$. Choose $z\in M$ so that $y$ is the reduction of~$z$  modulo~$c_nM$. By Proposition~\ref{propos_universal} there is a $\Gamma$-module homomorphism $\psi_z\colon\CU_n\to M$ such that $\psi_z(u)=z$. Then $y$ lies in the image of the homomorphism
 $$
 \psi_z\otimes(\Z/c_n\Z)\colon \CU_n/c_n\,\CU_n\to M/c_nM.
 $$
 Therefore, $\lambda y=y$ for all $\lambda\in\Lambda$.

 Finally, by condition~$(*)$ we have that the abelian group $$M/c_nM=(M/c_nM)_{\Lambda}=(M\otimes (\Z/c_n\Z))_{\Lambda}$$ is finitely generated, and hence finite.
\end{proof}

\section{Algebraicity of rational homology groups}\label{section_algebraic}

\begin{propos}\label{propos_alg_rep}
 Let $\bG$, $\Gamma$ and~$\CT$ be as in Theorem~\ref{thm_alg}. Assume in addition that $\CT$ consists of unipotent elements of~$\bG(\Q)$.  Then any $\CT$-unipotent $\Q[\Gamma]$-module~$V$ with $\dim_{\Q}(V)<\infty$ is algebraic.
\end{propos}

\begin{proof}
 Let $\rho\colon \Gamma\to\GL(V)$ be the representation of~$\Gamma$ on~$V$.
 By Corollary~\ref{cor_decompose_prod} we have $\rho(\gamma)=f(\gamma)u(\gamma)$, where the representation $f\colon \Gamma\to\GL(V)$ is algebraic, the representation $u\colon \Gamma\to\GL(V)$ has finite image~$u(\Gamma)$, and $f(\gamma_1)$ commutes with~$u(\gamma_2)$ for all $\gamma_1,\gamma_2\in\Gamma$. We need to show that the representation~$u$ is trivial. Consider an element $\tau\in\CT$. Since $\tau$ is a unipotent element of~$\bG(\Q)$ and $f$ is the restriction of a rational representation of~$\bG$, we obtain that $f(\tau)$ is a unipotent endomorphism. On the other hand, because the representation $\rho$ is $\CT$-unipotent, the endomorphism $\rho(\tau)$ is unipotent as well. It follows that  $u(\tau)$ is also unipotent. However, $u(\tau)$ has finite order in~$\GL(V)$. Hence $u(\tau)=1$. Since the set $\CT$ generates~$\Gamma$, we obtain that $u=1$, as required.
\end{proof}

\begin{proof}[Proof of Theorem~\ref{thm_alg_rep}]
From Theorems~\ref{thm_main} and~\ref{thm_annih} and Proposition~\ref{propos_Tav} it follows that the group $\Gamma=\Sp_{2g}(\Z)$, its subset~$\CT$ consisting of all symplectic transvections, and the $\Sp_{2g}(\Z)$-module~$H_k(\I_{g,p}^b;\Q)$ with $k\le g-2$ satisfy all conditions from Proposition~\ref{propos_alg_rep}. Consequently, the $\Sp_{2g}(\Z)$-representation $H_k(\I_{g,p}^b;\Q)$ is algebraic.

Let us prove the second assertion of Theorem~\ref{thm_alg_rep}. Suppose to the contrary that $H_k(\I_{g,p}^b;\Q)$ contains an irreducible subrepresentation~$\bV_{\lambda}=\bV_{\lambda_1,\ldots,\lambda_m}$ with $\lambda_1\ge k+1$. We use the standard notation from~\cite[Ch.~16]{FuHa91} for the Lie algebra $\fg=\fsp_{2g}(\Q)$ and the corresponding root system~$C_g$. Then the roots of~$\fg$ are the vectors~$\pm L_i\pm L_j$, where $1\le i,j\le g$, and $\bV_{\lambda}$ is a representation with the highest weight $\beta=\lambda_1L_1+\cdots+\lambda_nL_n$. Let $v$ be a highest weight vector in~$\bV_{\lambda}$. Consider the long root~$\alpha=-2L_1$ and the corresponding eigenvector $X_{\alpha}$ from the Chevalley basis for~$\fg$. (We have $X_{\alpha}=I+E_{n+1,1}$ in the standard  matrix representation of~$\fsp_{2g}(\Q)$.) Then the corresponding root unipotent
$$
t_{\alpha}=\exp(X_{\alpha})\in\Sp_{2g}(\Z)
$$
is a symplectic transvection. Since
$
\beta,\beta+\alpha,\ldots,\beta+\lambda_1\alpha
$
are weights of~$\bV_{\lambda}$ and $\beta+(\lambda_1+1)\alpha
$ is not a weight of~$\bV_{\lambda}$, we see that the vectors $v,X_{\alpha}v,\ldots,X_{\alpha}^{\lambda_1}v$ are nonzero, and $X_{\alpha}^{\lambda_1+1}v=0$. Consequently,
$$
(t_{\alpha}-1)^{\lambda_1}v=\bigl(\exp(X_{\alpha})-1\bigr)^{\lambda_1}v=X_{\alpha}^{\lambda_1}v\ne 0,
$$
which contradicts Theorem~\ref{thm_annih}. This completes the proof of Theorem~\ref{thm_alg_rep}.
\end{proof}

\section{Representation stability}\label{section_rs}

In this section, we present some constructions and results from the paper by Kupers and Randal-Williams~\cite{KuRW20} (with corrections from~\cite{KuRW26}) and then prove Theorem~\ref{thm_rs}. In~\cite{KuRW20},  higher‑dimensional Torelli groups were studied, namely the Torelli groups of diffeomorphisms of the manifolds~$\#^gS^n\times S^n$; we present all constructions and results only in the case that interests us, that is, the ordinary Torelli groups, i.\,e., $n=1$.

Set $H(g)=H^1(\Sigma_g^1;\Q)$, and let $(a,b)=\left\langle ab,\bigl[\Sigma_g^1,\partial \Sigma_g^1\bigr]\right\rangle$ be the standard skew-symmetric form on~$H(g)$. Throughout this section we use the fact that all rational representations of~$\Sp_{2g}(\Q)$ are self-dual;  in particular, the $\Sp_{2g}(\Q)$-representations~$H_k(\I_g^1;\Q)$ and~$H^k(\I_g^1;\Q)$ are isomorphic to each other in the stable range $k\le g-2$. Also we may not distinguish between irreducible algebraic representations of~$\Sp_{2g}(\Z)$ and irreducible rational representations of~$\Sp_{2g}(\Q)$, since $\Sp_{2g}(\Z)$ is Zariski dense in~$\Sp_{2g}(\Q)$.

We present the construction of graded spaces~$\CP(S)_{\ge0}$ and~$\CP(S)_{\ge0}'$ of certain labelled partitions of a finite set~$S$ and describe their relation to the twisted Miller--Morita--Mumford classes following~\cite{KuRW20}. A large part of this construction goes back to the work of Kawazumi~\cite{Kaw08}.

A \textit{labelled partition} of a finite set~$S$ is a finite collection~$P=\{(P_{\alpha},m_{\alpha})\}_{\alpha\in I}$ such that
\begin{itemize}
 \item $P_{\alpha}$ are  (possibly empty) subsets of~$S$ that are pairwise disjoint and whose union is~$S$,
 \item the labels~$m_{\alpha}$ are nonnegative integers.
\end{itemize}
We do not impose any ordering on the set of parts. Therefore, partitions that differ from each other by a reindexing of the parts are considered the same. However, partitions that differ by the addition of empty parts are considered different.

We assign to a finite set~$S$ the $\Q$-vector space~$\CP(S)_{\ge0}$ with basis the set of labelled partitions $P=\{(P_{\alpha},m_{\alpha})\}_{\alpha\in I}$ such that
\begin{enumerate}
 \item[(i)] each part of size~$0$ has label $\ge 3$,
 \item[(ii)] each part of size~$1$ has label $\ge 1$.
\end{enumerate}
Also, let~$\CP(S)'_{\ge0}$ be the $\Q$-vector space defined in the same way but with an additional restriction:
 \begin{enumerate}
  \item[(iii)] each part of size~$2$ has label $\ge 1$.
 \end{enumerate}
 Introduce gradings on the spaces~$\CP(S)_{\ge0}$ and~$\CP(S)'_{\ge0}$  by setting the degree of each labelled partition to $\sum_{\alpha\in I}(2m_{\alpha}+|P_{\alpha}|-2)$. In the case of~$\CP(S)'_{\ge0}$, the summand~$2m_{\alpha}+|P_{\alpha}|-2$ corresponding to each part~$P_{\alpha}$ is positive. In the case of~$\CP(S)_{\ge0}$, a partition may contain parts of size~$2$ and label~$0$ which provide zero impact to the degree.

There is  a natural map (see~\cite[Sect.~3]{KuRW20})
\begin{equation}\label{eq_PhiS}
\Phi_{g,S}\colon \CP(S)_{\ge0}\otimes\det\Q^S \to \left(H^*\bigl(\I_g^1;\Q)\otimes H(g)^{\otimes S}\right)^{\Sp_{2g}(\Z)}.
\end{equation}
This map can be written explicitly as follows.  Let $P=\{(P_{\alpha},m_{\alpha})\}_{\alpha\in I}$ be  a labelled partition  of~$S$ that satisfies~(i) and~(ii). To each part~$P_{\alpha}$ assign the element
\begin{multline}\label{eq_phi_alpha}
\varphi_{g,\alpha}=\sum_{j_1,\ldots,j_{|P_{\alpha}|}=1}^{2g}\kappa_{e^{m_{\alpha}}}\left(a_{j_1}\wedge\cdots\wedge a_{j_{|P_{\alpha}|}}\right)\otimes a_{j_1}^{\#}\otimes\cdots\otimes a_{j_{|P_{\alpha}|}}^{\#}\\{}\in H^{2m_{\alpha}+|P_{\alpha}|-2}\bigl(\I_g^1;\Q\bigr)\otimes H(g)^{\otimes|P_{\alpha}|}.
\end{multline}
Here $a_1,\ldots,a_{2g}$ is a basis of~$H(g)$ and $a_1^{\#},\ldots,a_{2g}^{\#}$ is the dual basis of~$H(g)$ characterized by $(a_i^{\#},a_j)=\delta_{ij}$. It is clear that the right-hand side of~\eqref{eq_phi_alpha} is independent of the choice of the basis $a_1,\ldots,a_{2g}$.
We then multiply all of the elements~$\varphi_{g,\alpha}$:
\begin{equation}\label{eq_PhiS_form}
\Phi_{g,S}(P)=\prod_{\alpha\in I}\varphi_{g,\alpha}\in \left(H^*\bigl(\I_g^1;\Q)\otimes H(g)^{\otimes |S|}\right)^{\Sp_{2g}(\Z)}.
\end{equation}
Note that $\varphi_{g,\alpha}$ have even degrees; hence, the result is independent of the order in which we take the product. Finally, we would like to replace~$H(g)^{\otimes |S|}$ with $H(g)^{\otimes S}$. This produces a sign
depending on the ordering of the elements of~$S$ up to an even permutation.
Thus, the map~$\Phi_{g,S}$ becomes well‑defined after adding the factor $\det\Q^S$ as in~\eqref{eq_PhiS}.

Now, set
$$
\Omega=\bigl[\kappa_{e^{2m}}\mid m>1\bigr]=\Q[e_3,e_5,e_7,\ldots].
$$
The space~$\CP(S)_{\ge0}$ is a graded $\Omega$-module with each~$\kappa_{e^{2m}}$ acting by adding an empty part with label~$m$. (The absence of~$e_1=\kappa_{e^2}$ in~$\Omega$ corresponds to the fact that empty parts with label~$2$ are not allowed.) Since the odd Miller--Morita--Mumford classes~$e_{2m-1}=\kappa_{e^{2m}}$ vanish in the rational cohomology of Torelli groups, the homomorphism~$\Phi_{g,S}$ induces a homomorphism

\begin{equation*}
\widehat{\Phi}_{g,S}\colon \Q\otimes_{\Omega}\CP(S)_{\ge0}\otimes\det\Q^S \longrightarrow \left(H^*\bigl(\I_g^1;\Q)\otimes H(g)^{\otimes S}\right)^{\Sp_{2g}(\Z)}.
\end{equation*}

We need the following three propositions.

\begin{propos}\label{propos_KuRW}
 The homomorphism~$\widehat{\Phi}_{g,S}$ is an isomorphism in the stable range of degrees satisfying $*+|S|\le \frac{2g-2}3$.
\end{propos}

Let $\fS_d$ denote the symmetric group, and let $S^{\lambda}$ denote the irreducible $\fS_d$-repre\-sen\-tation corresponding to a partition~$\lambda$ of~$d$.

\begin{propos}\label{propos_KuRW2}
 If $\lambda$ is a partition of~$d$, then in the stable range of degrees $*\le \frac{g-4}4$ the multiplicity of\/~$\bV_{\lambda}(g)$ in $H^k(\I_g^1;\Q)$ is equal to the multiplicity of~$S^{\lambda}$ in the $\fS_d$-representation on the $k$th graded component of
 \begin{equation}\label{eq_CP'}
 \Q\otimes_{\Omega}\CP([d])'_{\ge0}\otimes\det\Q^{[d]},
 \end{equation}
 where $[d]=\{1,\ldots,d\}$. In particular, if\/~$\bV_{\lambda}(g)$ is not defined, then $S^{\lambda}$ does not occur in~\eqref{eq_CP'}.
\end{propos}

\begin{propos}[\protect{\cite[Theorem~X.7]{KuRW26}}]\label{propos_KuRW3}
The vector space $\CP(S)_{\ge0}'$ vanishes in degree~$k$ as long as $|S|>3k$.
\end{propos}

Propositions~\ref{propos_KuRW} and~\ref{propos_KuRW2} were obtained by Kupers and Randal-Williams as \textit{conditional} results under the assumption that the cohomology groups~$H^k(\I_g^1;\Q)$ are finite-dimensional in the specified stable range. Moreover, Proposition~\ref{propos_KuRW2} was only a result about the algebraic part of~$H^k(\I_g^1;\Q)$, not about the whole~$H^k(\I_g^1;\Q)$. In view of our Theorems~\ref{thm_main} and~\ref{thm_alg_rep}, these propositions now become \textit{unconditional} results about the whole representations~$H^k(\I_g^1;\Q)$. As conditional results:
\begin{itemize}
 \item Proposition~\ref{propos_KuRW} is contained in Section~4 of~\cite{KuRW20} together with the correction of the stable range in~\cite[Proposition~X.1]{KuRW26},
 \item Proposition~\ref{propos_KuRW2} follows from Theorem~4.1 and Proposition~2.16 from~\cite{KuRW20} (see also the argument at the very beginning of Section~6), together with the corrections made in Theorem~4.1$'$ and Propositions~2.16$'$ from~\cite{KuRW26}.
\end{itemize}

Let $i_g\colon \Sigma_g^1\hookrightarrow\Sigma_{g+1}^1$ be the standard embedding. Then we have the pullback homomorphisms $i^*_g\colon H(g+1)\to H(g)$ and $i^*_g\colon H^1(\I_{g+1}^1;\Q)\to H^1(\I_g^1;\Q)$.

\begin{propos}\label{propos_theta}
 The diagram
 \begin{equation}\label{eq_cd_theta}
\begin{tikzcd}
 & \left(H^*\bigl(\I_{g+1}^1;\Q)\otimes H(g+1)^{\otimes S}\right)^{\Sp_{2g+2}(\Z)}
 \arrow[dd,"{i^*_g}"]\\
 \Q\otimes_{\Omega}\CP(S)_{\ge0}\otimes\det\Q^S
 \arrow[ur,"{\widehat{\Phi}_{g+1,S}}","{\cong}"']
 \arrow[dr,"{\cong}","{\widehat{\Phi}_{g,S}}"']
 & \\
 & \left(H^*\bigl(\I_g^1;\Q)\otimes H(g)^{\otimes S}\right)^{\Sp_{2g}(\Z)}
\end{tikzcd}
\end{equation}
is commutative. Hence the right vertical arrow~$i^*_g$ is an isomorphism in the stable range of degrees satisfying $*+|S|\le \frac{2g-2}3$.
\end{propos}

\begin{remark}
Note that this proposition is not as obvious as it seems to be. The problem is that
\begin{equation}\label{eq_kappa_restrict}
i^*_g\bigl(\kappa_{e^m}(v_1\wedge\cdots\wedge v_r)\bigr)=\kappa_{e^m}(i^*_gv_1\wedge\cdots\wedge i^*_gv_r)
\end{equation}
holds whenever $2m+r-2>0$ (see~\cite[Sect.~3.7]{KuRW20}), but this property obviously fails for the classes
$$
\kappa_1(v_1\wedge v_2)=(v_1,v_2)\in H^0(\I_{g+1}^1;\Q)=\Q,
$$
since the pullback $i^*_g\colon H(g+1)\to H(g)$ does not preserve the skew-symmetric form.
\end{remark}

\begin{proof}
 Consider a labelled partition $P=\{(P_{\alpha},m_{\alpha})\}_{\alpha\in I}$   of~$S$ that satisfies~(i) and~(ii). The commutativity of the diagram~\eqref{eq_cd_theta} will follow if we show that $i_g^*\varphi_{g+1,\alpha}=\varphi_{g,\alpha}$ for all~$\alpha\in I$. Since the right-hand side of the formula~\eqref{eq_phi_alpha} for~$\varphi_{g+1,\alpha}$  is independent of the choice of the basis $a_1,\ldots,a_{2g+2}$ of~$H(g+1)$, we may choose it so that:
\begin{itemize}
 \item the basis $a_1,\ldots,a_{2g+2}$ is symplectic, i.\,e., $(a_{2n-1},a_{2n})=1$ and $(a_j,a_k)=0$ for all other pairs of indices with $j<k$,
 \item $a_{2g+1},a_{2g+2}$ form a basis of~$\ker i_g^*$.
\end{itemize}
Then $a_{2n-1}^{\#}=-a_{2n}$ and $a_{2n}^{\#}=a_{2n-1}$, and $b_1=i^*_ga_1,\ldots,b_{2g}=i_g^*a_{2g}$ is a symplectic basis of~$H(g)$.

Firstly, assume that $P_{\alpha}$ is not a part of size~$2$ with label~$0$. Then the equality $i_g^*\varphi_{g+1,\alpha}=\varphi_{g,\alpha}$ follows immediately from~\eqref{eq_kappa_restrict}. Secondly, assume that $|P_{\alpha}|=2$ and $m_{\alpha}=0$. Then
\begin{align*}
\varphi_{g+1,\alpha}&=\sum_{j,k=1}^{2g+2}(a_j,a_k)a_j^{\#}\otimes a_k^{\#}\\{}&=2(a_1\otimes a_2-a_2\otimes a_1+\cdots+a_{2g+1}\otimes a_{2g+2}-a_{2g+2}\otimes a_{2g+1}),
\end{align*}
and similarly
$$
\varphi_{g,\alpha}=2(b_1\otimes b_2-b_2\otimes b_1+\cdots+b_{2g-1}\otimes b_{2g}-b_{2g}\otimes b_{2g-1}).
$$
Consequently,  $i_g^*\varphi_{g+1,\alpha}=\varphi_{g,\alpha}$.
\end{proof}

\begin{proof}[Proof of Theorem~\ref{thm_rs}]

Let us consecutively prove injectivity, multiplicity stablility, the bound for the degree $|\lambda|\le 3k$, and surjectivity. Note that the bound $\lambda_1\le k$ follows from Theorem~\ref{thm_alg_rep}.

\textsl{Injectivity.} By  Theorem~\ref{thm_explicit} the ring $H^*(\I_{g+1}^1;\Q)$ is generated in the stable range $*\le\frac{g-4}4$ by the classes $\kappa_{e^m}(v_1\wedge\cdots\wedge v_r)$ with $2m+r-2>0$. It then follows from~\eqref{eq_kappa_restrict} that the map $i_g^*\colon H^*(\I_{g+1}^1;\Q)\to H^*(\I_{g}^1;\Q)$ is surjective in this stable range. Consequently, the dual map in homology is injective.

\textsl{Multiplicity stability} in the stable range $*\le\frac{g-4}4$ follows immediately from Proposition~\ref{propos_KuRW2}, since the representation~\eqref{eq_CP'} is independent of~$g$.

\textsl{Bound for degrees.} Propositions~\ref{propos_KuRW2} and~\ref{propos_KuRW3} imply that the representation~$H_k(\I_g^1;\Q)$ with $g\ge 4k+4$ does not contain irreducible summands~$\bV_{\lambda}(g)$ of degree $|\lambda|>3k$.

\textsl{Surjectivity.} For $g\ge 4k+4$, we now have
$$
H_k\bigl(\I_g^1;\Q)\cong\bigoplus_{\lambda}c_{\lambda}\bV_{\lambda}(g),
$$
where the sum is taken over partitions of degree $\le 3k$ and the multiplicities~$c_{\lambda}$ are independent of~$g$. We need to prove that the map
$$
\Ind_{\Sp_{2g}(\Z)}^{\Sp_{2g+2}(\Z)}H_k\bigl(\I_g^1;\Q\bigr)\to H_k\bigl(\I_{g+1}^1;\Q\bigr)
$$
induced by~$i_{g*}$ is surjective for $g\ge 6k+1$. Suppose to the contrary that it is not. Then there is an irreducible $\Sp_{2g+2}(\Z)$-subrepresentation $W\subseteq H^k\bigl(\I_{g+1}^1;\Q\bigr)$ such that $i_g^*(W)=0$. By the above, $W\cong \bV_{\lambda}(g+1)$ for some partition~$\lambda$ of degree~$|\lambda|\le 3k$. Take a finite set~$S$ consisting of~$|\lambda|$ elements. Then $\bV_{\lambda}(g+1)$ occurs in the representation $H(g+1)^{\otimes S}$. Therefore, $\bigl(W\otimes H(g+1)^{\otimes S}\bigr)^{\Sp_{2g+2}(\Z)}\ne 0$. Consequently, the right vertical homomorphism in the diagram~\eqref{eq_cd_theta} has a nonzero kernel. On the other hand, since $k+|\lambda|\le \frac{2g-2}3$, Proposition~\ref{propos_theta} implies that this homomorphism is an isomorphism. This contradiction completes the proof of the surjectivity property.
\end{proof}

\section{Proof of Morita's conjecture}\label{section_inv}
In this section we prove Theorem~\ref{thm_inv}. Below the words `in the stable range' always mean `in the range of degrees $*\le \frac{2g-2}3$'.

We start with the case of~$\I_g^1$. Taking $S=\varnothing$ in Proposition~\ref{propos_KuRW}, we obtain an isomorphism in the stable range:
$$
\widehat{\Phi}_{g,\varnothing}\colon \Q\otimes_{\Omega}\CP(\varnothing)_{\ge0} \stackrel{\cong}{\longrightarrow} H^*\bigl(\I_g^1;\Q)^{\Sp_{2g}(\Z)}.
$$
The basis of~$\CP(\varnothing)_{\ge0}$ consists of all unordered finite collections of labels $\{m_{\alpha}\}$ with each $m_{\alpha}\ge 3$. Tensoring by~$\Q$ over~$\Omega$ kills all (empty) parts with even labels. Hence the basis of~$\Q\otimes_{\Omega}\CP(\varnothing)_{\ge0}$ consists of all collections $\{m_{\alpha}\}$ with each $m_{\alpha}=2n_{\alpha}+1$ odd and~$\ge 3$. Formula~\eqref{eq_PhiS_form} then reads as
$$
\widehat{\Phi}_{g,\varnothing}(\{m_{\alpha}\})=\prod_{\alpha}e_{2n_{\alpha}}.
$$
In the stable range, since~$\widehat{\Phi}_{g,\varnothing}$ is an  isomorphism, it follows that the ring~$H^*(\I_g^1;\Q)^{\Sp_{2g}(\Z)}$ is generated by the even Miller--Morita--Mumford classes, and moreover, all monomials in them are linearly independent, hence  the required equality
$$
H^*\bigl(\I_g^1;\Q\bigr)^{\Sp_{2g}(\Z)}=\Q[e_2,e_4,e_6,\ldots].
$$

We now consider the case of~$\I_{g,1}$. With some abuse of notation, we denote by~$e_k$ the Miller--Morita--Mumford classes in the cohomology of each of the three groups~$\I_g^1$, $\I_{g,1}$, and~$\I_g$.

We have a fibration sequence
\begin{equation*}
 B\I_g^1\stackrel{f}{\longrightarrow} B\I_{g,1}\stackrel{h}{\longrightarrow}B\SO(2).
\end{equation*}
The Leray--Hirsch property for this fibration sequence in the stable range on maximal algebraic subrepresetations was established in~\cite[Lemma~7.1]{KuRW20} under the condition that the rational cohomology groups are finite-dimensional in this stable rank. Since we now know from Theorems~\ref{thm_main} and~\ref{thm_alg_rep} that the rational cohomology are finite-dimensional and algebraic, we obtain the Leray--Hirsch property in the stable range in full generality. Consequently, we have an isomorphism of $H^*(B\SO(2);\Q)$-modules
 $$
 H^*(\I_{g,1};\Q)\cong H^*\bigl(\I_g^1;\Q\bigr)\otimes H^*(B\SO(2);\Q).
 $$
 Moreover, this isomorphism sends the class~$\epsilon\in H^2(\I_{g,1};\Q)$ (defined in Remark~\ref{rem_eps}) to $1\otimes u$, where $u\in H^2(B\SO(2);\Q)$ is the standard generator. This holds because $\epsilon$ restricts trivially to~$\I_g^1$ and $h$ is the classifying map for the bundle~$\eta$ from Remark~\ref{rem_eps}. It follows that $\epsilon$ is a regular element of~$H^*(\I_{g,1};\Q)$, i.\,e. not a zero-divisor in the stable range, and
 $$
 f^*\colon H^*(\I_{g,1};\Q)\twoheadrightarrow H^*\bigl(\I_g^1;\Q\bigr)
 $$
 is the quotient by the ideal~$(\epsilon)$.
Since the classes $\epsilon, e_2,e_4,\ldots$ in $H^*(\I_{g,1};\Q)$ are $\Sp_{2g}(\Z)$-invariant, we obtain the required equality
$$
H^*(\I_{g,1};\Q)=\Q[\epsilon,e_2,e_4,e_6,\ldots].
$$

 Finally, consider the case of~$\I_g$.
 We have a fibration sequence
 $$
 \Sigma_g\longrightarrow B\I_{g,1}\stackrel{\pi}{\longrightarrow} B\I_g.
 $$
 The pullback $$\pi^*\colon H^*(\I_g;\Q)\to H^*(\I_{g,1};\Q)$$ and the Gysin homomorphism $$\pi_!\colon H^*(\I_{g,1};\Q)\to H^{*-2}(\I_g;\Q)$$ are $\Sp_{2g}(\Z)$-equivariant. Moreover, $\pi_!(\epsilon)=\chi(\Sigma_g)=2-2g$ and hence $$\pi_!\bigl(\epsilon(\pi^*y)\bigr)=(2-2g)y$$ for all $y\in H^*(\I_g;\Q)$. Since we are working in the stable range $*\le \frac{2g-2}{3}$, we may assume that $g>1$. Then $\pi^*$ is injective and $\pi_!$ is surjective. Since we know that the cohomology groups considered are finite-dimensional algebraic representations of $\Sp_{2g}(\Z)$ and hence are semisimple, we see that
\begin{align*}
 H^*(\I_g;\Q)^{\Sp_{2g}(\Z)}&=\pi_!\left(H^*(\I_{g,1};\Q)^{\Sp_{2g}(\Z)}\right)\\
 {}&=\pi_!\left(\Q[\epsilon,e_2,e_4,e_6,\ldots]\right)\\
 {}&=\pi_!\left(\Q[e_2,e_4,e_6,\ldots]\right)\qquad\quad (\text{since }\pi_!(\epsilon)=2-2g)\\&=\Q[e_2,e_4,e_6,\ldots].
 \end{align*}
Moreover, the classes~$e_{2k}$ are algebraically independent in the stable range, since $\pi^*$ is a monomorphism.

\end{document}